\documentclass[12pt]{amsart}
\usepackage{}       %was 12pt
\usepackage{txfonts}
\usepackage{amssymb}
\usepackage{eucal}
\usepackage{graphicx}
\usepackage{makecell}
\usepackage{diagbox}
\usepackage{amsmath}
\usepackage{mathrsfs}
\usepackage{amscd}
\usepackage[all]{xy}           %xypic macro for latex2.09
\usepackage{amsfonts,latexsym}
\usepackage{xspace}
\usepackage{epsfig}
\usepackage{float}
\usepackage{color}
\usepackage{fancybox}
\usepackage{colordvi}
\usepackage{multicol}
\usepackage{colordvi}
\usepackage{tikz}
\usepackage{wasysym}
\usepackage[active]{srcltx} %SRC Specials for DVI Searching
\ifpdf
\usepackage[colorlinks,final,backref=page,hyperindex]{hyperref}
\else
\usepackage[colorlinks,final,backref=page,hyperindex,hypertex]{hyperref}
\fi
\newcommand{\nc}{\newcommand}
\newcommand{\delete}[1]{}	% use this to show deleted texts {\textcolor{yellow}{#1}}
\nc{\mlabel}[1]{\label{#1}}  % Use this to suppress names
\nc{\mcite}[1]{\cite{#1}}  % Use this to suppress names
\nc{\mref}[1]{\ref{#1}}  % Use this to suppress names
\nc{\meqref}[1]{\eqref{#1}} %
\nc{\mbibitem}[1]{\bibitem{#1}} % Use this to show number name
\delete{
\nc{\mlabel}[1]{\label{#1}  % Use the next two lines to show names
{\hfill \hspace{1cm}{\small\tt{{\ }\hfill(#1)}}}}
\nc{\mcite}[1]{\cite{#1}{\small{\tt{{\ }(#1)}}}}  % Use this lines to show names
\nc{\mref}[1]{\ref{#1}{{\tt{{\ }(#1)}}}}  % Use this lines to show names
\nc{\meqref}[1]{\eqref{#1}{{\tt{{\ }(#1)}}}}  % Use this lines to show names
\nc{\mbibitem}[1]{\bibitem[\bf #1]{#1}} % Use this to show name
}

\newtheorem{theorem}{Theorem}[section]

\newtheorem{prop}[theorem]{Proposition}

\theoremstyle{definition}
\newtheorem{defn}[theorem]{Definition}

\newtheorem{exam}[theorem]{Example}

\nc{\tred}[1]{\textcolor{red}{#1}}
\nc{\tblue}[1]{\textcolor{blue}{#1}}
\nc{\tgreen}[1]{\textcolor{green}{#1}}
\nc{\tpurple}[1]{\textcolor{purple}{#1}}
\nc{\btred}[1]{\textcolor{red}{\bf #1}}
\nc{\btblue}[1]{\textcolor{blue}{\bf #1}}
\nc{\btgreen}[1]{\textcolor{green}{\bf #1}}
\nc{\btpurple}[1]{\textcolor{purple}{\bf #1}}
\nc{\NN}{{\mathbb N}}
\nc{\ncsha}{{\mbox{\cyr X}^{\mathrm NC}}} \nc{\ncshao}{{\mbox{\cyr
X}^{\mathrm NC}_0}}

\newcommand{\efootnote}[1]{}
\renewcommand{\textbf}[1]{}
\nc{\opa}{\ast} \nc{\opb}{\odot} \nc{\op}{\bullet} \nc{\pa}{\frakL}
\nc{\arr}{\rightarrow} \nc{\lu}[1]{(#1)} \nc{\mult}{\mrm{mult}}
\nc{\diff}{\mathfrak{Diff}}
\nc{\opc}{\sharp}\nc{\opd}{\natural}
\nc{\ope}{\circ}
\nc{\dpt}{\mathrm{d}}
\nc{\tforall}{\text{ for all }}
\nc{\diam}{alternating\xspace}
\nc{\Diam}{Alternating\xspace}
\nc{\cdiam}{alternating\xspace}
\nc{\Cdiam}{Alternating\xspace}
\nc{\AW}{\mathcal{A}}

\nc{\ari}{\mathrm{ar}}

\nc{\lef}{\mathrm{lef}}

\nc{\Sh}{\mathrm{ST}}

\nc{\Cr}{\mathrm{Cr}}

\nc{\st}{{Schr\"oder tree}\xspace}
\nc{\sts}{{Schr\"oder trees}\xspace}

\nc{\vertset}{\Omega} % set of vertex decorations

\nc{\assop}{\quad \begin{picture}(5,5)(0,0)
\line(-1,1){10}
\put(-2.2,-2.2){$\bullet$}
\line(0,-1){10}\line(1,1){10}
\end{picture} \quad \smallskip}

\nc{\operator}{\begin{picture}(5,5)(0,0)
\line(0,-1){6}
\put(-2.6,-1.8){$\bullet$}
\line(0,1){9}
\end{picture}}

\nc{\idx}{\begin{picture}(6,6)(-3,-3)
\put(0,0){\line(0,1){6}}
\put(0,0){\line(0,-1){6}}
 \end{picture}}

\nc{\pb}{{\mathrm{pb}}}
\nc{\Lf}{{\mathrm{Lf}}}

\nc{\lft}{{left tree}\xspace}
\nc{\lfts}{{left trees}\xspace}

\nc{\fat}{{fundamental averaging tree}\xspace}

\nc{\fats}{{fundamental averaging trees}\xspace}
\nc{\avt}{\mathrm{Avt}}

\nc{\rass}{{\mathit{RAss}}}

\nc{\aass}{{\mathit{AAss}}}

\nc{\vin}{{\mathrm Vin}}    %decoration set of indices
\nc{\lin}{{\mathrm Lin}}    %decoration set of leaves
\nc{\inv}{\mathrm{I}n}
\nc{\gensp}{V} % space of generators
\nc{\genbas}{\mathcal{V}} % basis of the space of generators
\nc{\bvp}{V_P}     % Rota-Baxter generating space
\nc{\gop}{{\,\omega\,}}     % generic binary operation

\nc{\bin}[2]{ (_{\stackrel{\scs{#1}}{\scs{#2}}})}  %binomial coeff
\nc{\binc}[2]{ \left (\!\! \begin{array}{c} \scs{#1}\\
    \scs{#2} \end{array}\!\! \right )}  %binomial coeff
\nc{\bincc}[2]{  \left ( {\scs{#1} \atop
    \vspace{-1cm}\scs{#2}} \right )}  %binomial coeff
\nc{\bs}{\bar{S}} \nc{\cosum}{\sqsubset} \nc{\la}{\longrightarrow}
\nc{\rar}{\rightarrow} \nc{\dar}{\downarrow} \nc{\dprod}{**}
\nc{\dap}[1]{\downarrow \rlap{$\scriptstyle{#1}$}}
\nc{\md}{\mathrm{dth}} \nc{\uap}[1]{\uparrow
\rlap{$\scriptstyle{#1}$}} \nc{\defeq}{\stackrel{\rm def}{=}}
\nc{\disp}[1]{\displaystyle{#1}} \nc{\dotcup}{\
\displaystyle{\bigcup^\bullet}\ } \nc{\gzeta}{\bar{\zeta}}
\nc{\hcm}{\ \hat{,}\ } \nc{\hts}{\hat{\otimes}}
\nc{\barot}{{\otimes}} \nc{\free}[1]{\bar{#1}}
\nc{\uni}[1]{\tilde{#1}} \nc{\hcirc}{\hat{\circ}} \nc{\lleft}{[}
\nc{\lright}{]} \nc{\lc}{\lfloor} \nc{\rc}{\rfloor}
\nc{\curlyl}{\left \{ \begin{array}{c} {} \\ {} \end{array}
    \right .  \!\!\!\!\!\!\!}
\nc{\curlyr}{ \!\!\!\!\!\!\!
    \left . \begin{array}{c} {} \\ {} \end{array}
    \right \} }
\nc{\longmid}{\left | \begin{array}{c} {} \\ {} \end{array}
    \right . \!\!\!\!\!\!\!}
\nc{\onetree}{\bullet} \nc{\ora}[1]{\stackrel{#1}{\rar}}
\nc{\ola}[1]{\stackrel{#1}{\la}}%${\Bbb Z}$
\nc{\ot}{\otimes} \nc{\mot}{{{\boxtimes\,}}}
\nc{\otm}{\overline{\boxtimes}} \nc{\sprod}{\bullet}
\nc{\scs}[1]{\scriptstyle{#1}} \nc{\mrm}[1]{{\rm #1}}
\nc{\margin}[1]{\marginpar{\rm #1}}   %{\rm #1}}
\nc{\dirlim}{\displaystyle{\lim_{\longrightarrow}}\,}
\nc{\invlim}{\displaystyle{\lim_{\longleftarrow}}\,}
\nc{\mvp}{\vspace{0.3cm}} \nc{\tk}{^{(k)}} \nc{\tp}{^\prime}
\nc{\ttp}{^{\prime\prime}} \nc{\svp}{\vspace{2cm}}
\nc{\vp}{\vspace{8cm}} \nc{\proofbegin}{\noindent{\bf Proof: }}
\nc{\proofend}{$\blacksquare$ \vspace{0.3cm}}
\nc{\modg}[1]{\!<\!\!{#1}\!\!>}
\nc{\intg}[1]{F_C(#1)} \nc{\lmodg}{\!
<\!\!} \nc{\rmodg}{\!\!>\!}
\nc{\cpi}{\widehat{\Pi}}
\nc{\sha}{{\mbox{\cyr X}}}  %used to be \cyr
\nc{\shap}{{\mbox{\cyrs X}}} %sha as product
\nc{\shan}{{\overrightarrow \sha}}
\nc{\shpr}{\diamond}    %Shuffle product
\nc{\shp}{\ast} \nc{\shplus}{\shpr^+}
\nc{\shprc}{\shpr_c}    %Cartier's product
\nc{\msh}{\ast} \nc{\zprod}{m_0} \nc{\oprod}{m_1}
\nc{\vep}{\varepsilon} \nc{\labs}{\mid\!} \nc{\rabs}{\!\mid}
\nc{\sqmon}[1]{\langle #1\rangle}
\nc{\mmbox}[1]{\mbox{\ #1\ }} \nc{\dep}{\mrm{dep}} \nc{\fp}{\mrm{FP}}
\nc{\rchar}{\mrm{char}} \nc{\End}{\mrm{End}} \nc{\Fil}{\mrm{Fil}}
\nc{\Mor}{Mor\xspace} \nc{\gmzvs}{gMZV\xspace}
\nc{\gmzv}{gMZV\xspace} \nc{\mzv}{MZV\xspace}
\nc{\mzvs}{MZVs\xspace} \nc{\Hom}{\mrm{Hom}} \nc{\id}{\mrm{id}}
\nc{\im}{\mrm{im}} \nc{\incl}{\mrm{incl}} \nc{\map}{\mrm{Map}}
\nc{\mchar}{\rm char} \nc{\nz}{\rm NZ} \nc{\supp}{\mathrm Supp}

\nc{\Alg}{\mathbf{Alg}} \nc{\Bax}{\mathbf{Bax}} \nc{\bff}{\mathbf f}
\nc{\bfk}{{\bf k}} \nc{\bfone}{{\bf 1}} \nc{\bfx}{\mathbf x}
\nc{\bfy}{\mathbf y}
\nc{\base}[1]{\bfone^{\otimes ({#1}+1)}} %{{a_{#1}}}
\nc{\Cat}{\mathbf{Cat}}

\nc{\detail}{\marginpar{\bf More detail}
    \noindent{\bf Need more detail!}
    \svp}
\nc{\Int}{\mathbf{Int}} \nc{\Mon}{\mathbf{Mon}}
\nc{\rbtm}{{shuffle }} \nc{\rbto}{{Rota-Baxter }}
\nc{\remarks}{\noindent{\bf Remarks: }} \nc{\Rings}{\mathbf{Rings}}
\nc{\Sets}{\mathbf{Sets}} \nc{\wtot}{\widetilde{\odot}}
\nc{\wast}{\widetilde{\ast}} \nc{\bodot}{\bar{\odot}}
\nc{\bast}{\bar{\ast}} \nc{\hodot}[1]{\odot^{#1}}
\nc{\hast}[1]{\ast^{#1}} \nc{\mal}{\mathcal{O}}
\nc{\tet}{\tilde{\ast}} \nc{\teot}{\tilde{\odot}}
\nc{\oex}{\overline{x}} \nc{\oey}{\overline{y}}
\nc{\oez}{\overline{z}} \nc{\oef}{\overline{f}}
\nc{\oea}{\overline{a}} \nc{\oeb}{\overline{b}}
\nc{\weast}[1]{\widetilde{\ast}^{#1}}
\nc{\weodot}[1]{\widetilde{\odot}^{#1}} \nc{\hstar}[1]{\star^{#1}}
\nc{\lae}{\langle} \nc{\rae}{\rangle}
\nc{\lf}{\lfloor}
\nc{\rf}{\rfloor}
\nc{\QQ}{{\mathbb Q}}
\nc{\RR}{{\mathbb R}} \nc{\ZZ}{{\mathbb Z}}

\nc{\cala}{{\mathcal A}} \nc{\calb}{{\mathcal B}}
\nc{\calc}{{\mathcal C}}
\nc{\cald}{{\mathcal D}} \nc{\cale}{{\mathcal E}}
\nc{\calf}{{\mathcal F}} \nc{\calg}{{\mathcal G}}
\nc{\calh}{{\mathcal H}} \nc{\cali}{{\mathcal I}}
\nc{\call}{{\mathcal L}} \nc{\calm}{{\mathcal M}}
\nc{\caln}{{\mathcal N}}\nc{\calo}{{\mathcal O}}
\nc{\calp}{{\mathcal P}} \nc{\calr}{{\mathcal R}}
\nc{\cals}{{\mathcal S}} \nc{\calt}{{\mathcal T}}
\nc{\calu}{{\mathcal U}} \nc{\calw}{{\mathcal W}} \nc{\calk}{{\mathcal K}}
\nc{\calx}{{\mathcal X}} \nc{\CA}{\mathcal{A}}

\nc{\fraka}{{\mathfrak a}} \nc{\frakA}{{\mathfrak A}}
\nc{\frakb}{{\mathfrak b}} \nc{\frakB}{{\mathfrak B}}
\nc{\frakD}{{\mathfrak D}} \nc{\frakF}{\mathfrak{F}}
\nc{\frakf}{{\mathfrak f}} \nc{\frakg}{{\mathfrak g}}
\nc{\frakH}{{\mathfrak H}} \nc{\frakL}{{\mathfrak L}}
\nc{\frakM}{{\mathfrak M}} \nc{\bfrakM}{\overline{\frakM}}
\nc{\frakm}{{\mathfrak m}} \nc{\frakP}{{\mathfrak P}}
\nc{\frakN}{{\mathfrak N}} \nc{\frakp}{{\mathfrak p}}
\nc{\frakS}{{\mathfrak S}} \nc{\frakT}{\mathfrak{T}}
\nc{\frakX}{{\mathfrak X}} \nc{\frakx}{\mathfrak{x}}

\nc{\BS}{\mathbb{S
}}

\font\cyr=wncyr10 \font\cyrs=wncyr7
\nc{\li}[1]{\textcolor{red}{#1}}
\nc{\lir}[1]{\textcolor{red}{Li:#1}}

\nc{\sz}[1]{\textcolor{blue}{SZ: #1}}

\nc{\qhz}[1]{\textcolor{green}{Huizhen: #1}}

\nc{\ID}{\mathfrak{I}} \nc{\lbar}[1]{\overline{#1}}
\nc{\bre}{{\rm b}} \nc{\sd}{\cals} \nc{\rb}{\rm RB}
\nc{\A}{\rm angularly decorated\xspace} \nc{\LL}{\rm L}
\nc{\w}{\rm wid} \nc{\arro}[1]{#1}
\nc{\ver}{\rm ver}
\nc{\FN}{F_{\mathrm N}}
\nc{\FNA}{\FN(A)} \nc{\NA}{N_{A}}
\nc{\dr}{\diamond_r}
\nc{\dia}{\diamond}
\nc{\shar}{{\mbox{\cyrs X}}_r} %shar as right-shift shuffle product
\nc{\dt}{\Delta_T}
\nc{\da}{\Delta_A}
\nc{\vt}{\vep_T }
\nc{\bul}{\bullet}
\nc{\fraku}{{\mathfrak U}}
\nc{\frakc}{{\mathfrak C}}
\nc{\bp}{\bar{P}}
\nc{\qrba}{quasi-idempotent Rota-Baxter algebra\xspace}
\nc{\qrbo}{quasi-idempotent Rota-Baxter operator\xspace}
\nc{\rba}{Rota-Baxter algebra\xspace}
\nc{\rbo}{Rota-Baxter operator\xspace}
\nc{\bd}{\bar{\diamond}}
\nc{\Id}{\mathrm{Id}}
\nc{\Irr}{\mathrm{Irr}}
\nc{\brea}{\mrm{bre}}
\nc{\bro}{\mathrm{br\omega}}
\nc{\dgo}{\mathrm{dg\omega}}
\nc{\dgy}{\mathrm{dgy}}
\nc{\db}{\mathrm{cdb}}
\nc{\dlex}{\mathrm{dlex}}
\nc{\dao}{\overline{A^{\NN}}}
\nc{\rp}{{\rm p}}
\nc{\rd}{{\rm d}}

\begin{document}

\title[Construction of free quasi-idempotent differential Rota-Baxter algebras by GS bases]{Construction of free quasi-idempotent differential Rota-Baxter algebras by Gr\"obner-Shirshov bases}
%
%=========================================================================
\author{Huizhen Qiu}
\address{School of Mathematics and Statistics, Jiangxi Normal University, Nanchang, Jiangxi 330022, China}
\email{1197147595@qq.com}

\author{Shanghua Zheng}
\address{School of Mathematics and Statistics, Jiangxi Provincial Center for Applied Mathematics, Jiangxi Normal University, Nanchang, Jiangxi 330022, China}
\email{zhengsh@jxnu.edu.cn}

\author{Yangfan Dan}
\address{School of Mathematics and Statistics, Jiangxi Normal University, Nanchang, Jiangxi 330022, China}
\email{1732882275@qq.com}

%========================================================================
%========================================================================
\date{\today}
%========================================================================
\begin{abstract}Differential operators and integral operators are linked together by the first fundamental theorem of calculus. Based on this principle, the notion of a differential Rota-Baxter algebra was proposed by Guo and Keigher from an algebraic abstraction point of view. Recently, the subject has attracted more attention since it is associated with many areas in mathematics, such as integro-differential algebras. This paper considers differential algebras, Rota-Baxter algebras and differential Rota-Baxter algebras in the quasi-idempotent operator context. We establish a Gr\"obner-Shirshov basis for free commutative quasi-idempotent differential algebras (resp. Rota-Baxter algebras, resp. differential Rota-Baxter algebras).  This provides a linear basis of  free object in each of the three corresponding categories by Composition-Diamond lemma.
\end{abstract}

\subjclass[2010]{
13A99, %Commutative algebra		General commutative ring theory
16W99, %Rings and algebras with additional structure
13P10, %Grobner bases; other bases for ideals and modules
%16S10, %Rings determined by universal properties (free algebras, coproducts, adjunction of inverses, etc.)
08B20, %free algebras
12H05 %Differential algebra%45N05 %Abstract integral equations, integral equations in abstract spaces
}

\keywords{Rota-Baxter algebra; Differential algebra; Differential Rota-Baxter algebra; Gr\"obner-Shirshov basis; Quasi-idempotent operator.}

\maketitle

\tableofcontents

\setcounter{section}{0}

\allowdisplaybreaks

%========================================================================
\section{Introduction}
The aim of this paper is to give linear bases of free commutative quasi-idempotent differential algebras
(resp. Rota-Baxter algebras, resp. differential Rota-Baxter algebras) by means of  Gr\"obner-Shirshov bases.

\subsection{Differential Rota-Baxter algebras}
A {\bf differential operator of weight $\lambda$} is a linear operator $d$ on a unitary $\bfk$-algebra $R$ such that
\begin{equation}\mlabel{eq:gle}
d(xy)=d(x)y+xd(y)+\lambda d(x)d(y),\quad\forall x,y \in R
\end{equation}
and
\begin{equation}\mlabel{eq:done}
d(1)=0.
\end{equation}
When $\lambda=0$, it becomes the classical differential operator, which is essential in analysis and its algebraic generalizations led to the development of differential algebra in 1950s~\cite{Ritt}. Over the next few decades, the theory was completely developed, including differential Galois theory and differential algebraic groups~\cite{Kol,Magid}.

A {\bf Rota-Baxter operator of weight $\lambda$} is a linear operator $P$ on $R$ such that
\begin{equation}\mlabel{eq:rb}
  P(x)P(y)=P(xP(y))+P(P(x)y)+\lambda P(xy),\quad\forall x,y\in R.
\end{equation}
When $\lambda=0$, a Rota-Baxter operator is just the integral operator in calculus. The Rota-Baxter algebra was originated from the probability study of G. Baxter in 1960~\cite{Ba}, and was developed further by many outstanding mathematicians such as Atkinson,  Cartier and Rota~\cite{At,Ca,Ro,Ro2} later.  In the early 1980s,
Rota-Baxter operators in the context of Lie algebras were discovered by Belavin and Drinfeld in~\cite{BD82}, and Semenov-Tian-Shansky in~\cite{STS} as the operator form of the classical Yang-Baxter equation.
Since then, the study of associative Rota-Baxter algebras has undergone an extraordinary renascence in the application of mathematics and mathematical physics, such as Yang-Baxter equations~\cite{Bai,BBGN}, shuffle products~\cite{GK00,Hof00}, combinatorics~\cite{Ca,Ro}, Lie groups~\cite{GLS21,LS22} and Hopf algebras~\cite{EG06,YGT}. Most significantly, there exists the Connes-Kreimer's algebraic approach to the renormalization in perturbative quantum field theory~\cite{CK00}. For further details on Rota-Baxter algebras, see~\cite{Gub}.

A differential Rota-Baxter algebra of weight $\lambda$ (see Definition~\mref{defn:qidrb}) contains both a differential operator of weight $\lambda$ and a Rota-Baxter operator of weight $\lambda$ that are related in the same way as differential operators and  integral operators. In 2008, Guo and Keigher introduced the differential Rota-Baxter algebra of weight $\lambda$, and obtained explicit constructions of free commutative differential Rota-Baxter algebras by mixable shuffles and  free noncommutative differential Rota-Baxter algebras by angularly decorated rooted trees~\cite{GK08}. In fact, the explicit constructions of free commutative and noncommutative Rota-Baxter algebras could be obtained by means of the same method~\cite{EG081,EG08,Gub}. It is also worth mentioning that a differential Rota-Baxter algebras includes a integro-differential algebra as a special case, which was proposed due to the algebraic study of boundary problems for linear ordinary differential equations~\cite{GRR14}. As an essential algebraic structure in differential Rota-Baxter algebras, the ring of differential Rota-Baxter operators analogous to the ring of differential operators,  was studied in~\cite{GGR18}.  The algebra of integro-differential operators was also developed in~\cite{Bav11}. In 2016,  from a categorical point of view, Zhang, Guo and Keigher studied monads on differential algebras and comonads on Rota-Baxter algebras, providing the mixed distributive law to differential Rota-Baxter algebras~\cite{ZGK16}.

\subsection{Gr\"obner-Shirshov bases for free $\Omega$-operated algebras}
The work of Shirshov~\cite{Shir} raised a crucial question of Lie algebras: how to find out a linear basis of a Lie algebra presented by generators and relations?
The method of Gr\"obner-Shirshov (aka. GS) bases provides an infinite algorithm to solve this question faultlessly. Now it has become a  very popular tool to study the constructions of free objects in a wide range of algebraic structures, such as semigroups, algebras and operads. For further details of GS bases, see the survey~\cite{BC14}.

In particular, the GS bases  theory for free commutative and noncommutative $\Omega$-operated algebras has been successfully established in~\cite{BCQ10,Qiu14}. Later, they were used to construct the free commutative and noncommutative differentia Rota-Baxter algebras of weight $\lambda$. Furthermore,  GS bases were applied to the explicit construction of free object in the category of differential type algebras, Rota-Baxter type algebras, integro-differential algebras and Lie differential Rota-Baxter algebras~\cite{BCD,GGZ14,GGM15,LG21,QC10,QC171,QQWZ21}.

This paper is devoted to refine some aspects of the above works in the quasi-idempotent operators context by employing the general method of GS bases.

\subsection{Quasi-idempotent operators and outline of the paper}There are intimate relations among quasi-idempotent operators, Rota-Baxter algebras and Hopf algebras.
Jian~\cite{Jian17} constructed the Rota-Baxter operator by using quasi-idempotent elements. Based on this result, it follows that each finite dimensional Hopf algebra possesses a Rota-Baxter algebra structure. Dually, Ma, Li and Yang~\cite{MLY21} proved that every finite dimensional Hopf algebra  contains a Rota-Baxter coalgebra (bialgebra) by the same way. Recently, Zheng, Guo and Zhang~\cite{ZGZ20} characterized a quasi-idempotent operator by a generic Rota Baxter paired modules of weight $\lambda$.

On the other hand, Aguiar and  Moreira~\cite{AM} gave the explicit construction of free noncommutative quasi-idempotent Rota-Baxter algebras on one generator by using angularly decorated rooted trees. More importantly,  the free tridendriform (resp. dendriform) algebra on one generator can be embedded into the free quasi-idempotent Rota-Baxter algebras (resp. with an idempotent generator).

Motivated by a differential operator being a left inverse of a Rota-Baxter operator, it is desirable to introduce a notion of a quasi-idempotent differential algebra and consider the free object in the corresponding category. Consequently differential operators, Rota-Baxter operators and quasi-idempotent operators combine together  to provide a quasi-idempotent differential Rota-Baxter algebra. We shall gives an answer to the Shirshov's question of free quasi-idempotent differential Rota-Baxter algebras by GS bases, with emphasis on the commutative case.

This is our main motivation of this paper.

The outline of this paper is as follows.
Section~\mref{sec:bsp} gives some properties and examples of quasi-idempotent differential Rota-Baxter algebras.  Section~\mref{sec:foa} recalls  the explicit  construction of free commutative $\Omega$-operated algebras on a set, and establishes the Composition-Diamond lemma for free commutative $\Omega$-operated algebras. A monomial order on free commutative $\Omega$-operated monoids is given for GS bases for free commutative quasi-idempotent differential Rota-Baxter algebras. With these results in hand,
Section~\mref{sec:fcqa} first obtains a GS basis for free commutative quasi-idempotent differential (resp. Rota-Baxter) algebras of weight $\lambda$.  Based on this,  we supply a  GS basis for free commutative quasi-idempotent differential Rota-Baxter algebras of weight $\lambda$.  As a consequence,  a linear basis of the free commutative quasi-idempotent differential algebras (resp. Rota-Baxter algebras, resp. differential Rota-Baxter algebras) is obtained.

\medskip\noindent
{\bf Convention.} Throughout this paper, $\bfk$ is taken to be a field of characteristic $0$. All algebras and linear maps are taken over $\bfk$, unless the contrary is specified. By an algebra we mean an associative unitary \bfk-algebra.

\section{Basic properties and examples}
\mlabel{sec:bsp}
We first introduce notions of a quasi-idempotent differential algebra, a quasi-idempotent Rota-Baxter algebra and a quasi-idempotent differential Rota-Baxter algebra. Then some examples of each of these three algebras are given, respectively.

\begin{defn}\mlabel{defn:qidrb}
Let $R$ be an algebra. Let $0\neq \lambda \in \mathbf{k}$ be fixed.
\begin{enumerate}
\item A linear operator $P: R\to R$ is called a {\bf quasi-idempotent operator of weight $\lambda$} on $R$ if $P^2=-\lambda P$.
\item A nonzero element $\xi \in R $ is called a {\bf quasi-idempotent element of weight $\lambda$} if $\xi^2=-\lambda \xi$.
\item A {\bf quasi-idempotent differential algebra of weight $\lambda$} is an algebra $R$ together with a quasi-idempotent operator $d$ of weight  $\lambda^{-1}$ (that is, $d^2=-\lambda^{-1} d$) on $R$ satisfying Eqs.~(\mref{eq:gle}) and~(\mref{eq:done}).
Then $d$ is called a {\bf quasi-idempotent differential operator of weight $\lambda$}.
\item A {\bf quasi-idempotent Rota-Baxter algebra of weight $\lambda$} is an algebra $R$ together with a quasi-idempotent operator $P$ of weight  $\lambda$ on $R$ satisfying Eq.~(\mref{eq:rb}).
Then $P$ is called a {\bf quasi-idempotent Rota-Baxter operator of weight $\lambda$}.
\item A {\bf quasi-idempotent differential Rota-Baxter algebra of weight $\lambda$} is an algebra $R$ together with a quasi-idempotent differential operator $d$ of weight $\lambda$ and a quasi-idempotent Rota-Baxter operator $P$ of weight $\lambda$ such that
\begin{equation}
  d\circ P=\id_R.
\end{equation}

\item
If a quasi-idempotent operator $d$ of weight $\lambda^{-1}$ exactly satisfies  Eq.~(\mref{eq:gle}), then $d$ is called a {\bf weak quasi-idempotent differential operator of weight $\lambda$}. Particularly, a weak quasi-idempotent differential operator of weight $\lambda$ with $d(1)\neq 0$ is called a {\bf degenerate quasi-idempotent differential operator of weight $\lambda$}.
\end{enumerate}
\end{defn}
For simplicity,  we will also use the notion of a quasi-idempotent operator (resp. element, resp. algebra)  interchangeably with the notion a quasi-idempotent operator (resp. element, resp. algebra) of weight $\lambda$.
We next give some properties of quasi-idempotent Rota-Baxter algebras.
\begin{prop}\mlabel{prop:tilp}
Let $P$ be a \qrbo of weight $\lambda$. Then $\tilde{P}:=-\lambda \id-P$ is also a \qrbo of weight of $\lambda$.
\end{prop}
\begin{proof}By \cite[Proposition~1.1.12]{Gub}, $\tilde{P}$ is  a \rbo of weight $\lambda$.
By $P^2(x)=-\lambda P(x)$,
$$\tilde{P}^2(x)=\lambda ^2 x+ 2\lambda P(x)+ P^2(x)=(-\lambda) (-\lambda \id-P)(x)=-\lambda \tilde{P}(x).$$
Thus $\tilde{P}$ is a quasi-idempotent operator of weight $\lambda$.
\end{proof}
\begin{prop}Let $P$ be a linear operator on an algebra $R$. Then $P$  is a \qrbo of nonzero weight $\lambda$ if and only if there exists a  vector space direct sum decomposition
$$R=R_1\oplus R_2$$
of $R$ into $\bfk$-subalgebras $R_1$ and $R_2$ of  $R$ such that for all $a:=a_1+a_2$ with $ a_1\in R_1, a_2\in R_2$,
$$P:R\to R_1, \,a\mapsto -\lambda a_1.$$
\end{prop}
\begin{proof}Suppose that $P$ is a quasi-idempotent Rota-Baxter operator of nonzero weight $\lambda$. Then by Proposition~\mref{prop:tilp}, $\tilde{P}=-\lambda\id-P$ is also a \qrbo of weight $\lambda$.   Let $R_1:=P(R)$ and let $R_2=\tilde{P}(R)$.
Then by \cite[Proposition~1.1.12]{Gub}, $R_1$ and $R_2$ are subalgebras of $R$. We next verify that $R=R_1\oplus R_2$. Firstly, we have
$$a=-\lambda^{-1}P( a)+(-\lambda^{-1})(-\lambda a-P(a)).$$
Thus $R=R_1+R_2$.  Furthermore, we let $x\in R_1\cap R_2$. Then for some $r_1,r_2\in R$,
$$x=P(r_1)=-\lambda r_2-P(r_2).$$
Thus
$$x=-\lambda^{-1}P^2(r_1)=(-\lambda^{-1})P(-\lambda r_2-P(r_2))=P(r_2)-P(r_2)=0.$$
Thus $R=R_1\oplus R_2$.  Assume that $a=a_1+a_2$ for some $a_1\in R_1$ and $a_2\in R_2$. Since $a=-\lambda^{-1}P( a)+(-\lambda^{-1})(-\lambda a-P(a))$, we know that $P(a)=-\lambda a_1$.

Conversely, by direct computation, we get
$$P^2(a)=P(-\lambda a_1)=\lambda^2 a_1=-\lambda P(a).$$
From the proof of \cite[Theorem~1.1.14]{Gub}, we obtain that $P$ is a \rbo of weight  $\lambda$. Thus $P$  is a \qrbo of nonzero weight $\lambda$.
\end{proof}
The intimate relationship between quasi-idempotent Rota-Baxter algebras and Nijenhuis algebras is presented as below.

\begin{defn}A {\bf Nijenhuis algebra} is a pair $(R,P)$ consisting of an algebra $R$ and a linear operator $P$ on $R$ satisfying the {\bf Nijenhuis identity}
\begin{equation}
P(x)P(y)=P(xP(y))+P(P(x)y)-P^2(xy),\quad\forall x,y\in R.
\mlabel{eq:nij}
\end{equation}
\end{defn}

\begin{prop}Let $(R,P)$ be a \qrba of weight $\lambda$. Then $(R,P)$ is a Nijenhuis algebra.
\end{prop}
\begin{proof}
Since $P$ is a \rbo of weight $\lambda$, we have
$P(x)P(y)=P(xP(y))+P(P(x)y)+\lambda P(xy)$.
Then by $P^2(xy)=-\lambda P(xy)$,  we get
$P(x)P(y)=P(xP(y))+P(P(x)y)-P^2(xy)$,
proving the statement.
\end{proof}
But the converse of the above proposition is false.  For example,
by \cite[Example~2.3]{Ebr14},  the left  multiplication operator $L_a: R\to R, x\mapsto ax$ is a Nijenhuis operator.
If $a^2\neq -\lambda a$,  then $L_a^2(x)=a^2x\neq -\lambda ax=-\lambda L_a(x)$ for some $x\in R$. Thus $L_a$ is not quasi-idempotent, and so is not a \qrbo of weight $\lambda$.

Next we consider some examples of quasi-idempotent differential Rota-Baxter algebras. By~\cite[Proposition~2.10]{GK08}, there is a classical differential Rota-Baxter algebra structure on $\lambda$-Hurwitz series. It is desirable to consider the quasi-idempotent case. For this, we recall the notion of $\lambda$-Hurwitz series.
Let $A$ be a $\bfk$-algebra.  Denote
$$A^{\NN}:=\{f:\NN\to A\}:=\{(f(n)):=(f(0),f(1),f(2),\cdots,f(n),\cdots)\,|\,f(n)\in A,\,n\geq 0\}.$$
Addition and scalar multiplication on $A^{\NN}$ are defined by: for all $f,g\in A^{\NN}$,
\begin{eqnarray*}
&&f+g=((f+g)(n)),\quad \text{where}\quad (f+g)(n)=f(n)+g(n)\\
&&kf=((kf)(n)), \,\quad \text{where}\quad (kf)(n)=kf(n).
\end{eqnarray*}
The product on $A^\NN$ is defined by:
\begin{equation*}
fg=((fg)(n)),\quad\text{where}\quad(fg)(n)=\sum_{k=0}^n\sum_{j=0}^{n-k}{n\choose k}{n-k\choose j}\lambda^k f(n-j)g(k+j).
\end{equation*}
Then $A^\NN$ becomes a unitary commutative $\bfk$-algebra, called the {\bf $\lambda$-Hurwitz series algebra over $A$} or simply the {\bf $\lambda$-Hurwitz series ring over $A$}, also denoted by $A^\NN$. The identity $1$ of $A^\NN$ is given by $1(0)=1_A$ and $1(n)=0$ if $n>0$.
Define a linear map $$\partial_{A}:A^\NN\to A^\NN, \quad (\partial_{A}(f))(n)=f(n+1), \quad n\in \NN,\,f\in A^\NN.$$
%\qhz{$$\bar{\partial}_{A}:A^\NN\to A^\NN, \quad ((\bar{\partial}_{A}(f))(n))=(f(n+1)), \quad n\in \NN,\,f\in A^\NN.$$}
Then by \cite[Proposition 2.7]{GK08}, $(A^\NN,\partial_A)$ is a differential algebra of weight $\lambda$.
If $0\neq \lambda\in\bfk$, we  set
$$ \overline{A^\NN}=\{f:\NN\to A \mid  f(n)=-\lambda^{-1} f(n-1), n\geq 1\}\subseteq A^\NN.$$
Thus we can write $f\in \overline{A^\NN}$ as
\begin{equation*}
f=(f(0),-\lambda^{-1}f(0),(-\lambda^{-1})^2f(0),\cdots (-\lambda^{-1})^{n}f(0),\cdots).
\end{equation*}
Denote by $\bar{\partial}_A$  the restriction of $\partial_A$ to $\overline{A^\NN}$.
A direct computation shows that $\dao$ together with $\bar{\partial}_A$ forms a quasi-idempotent differential subalgebra of $A^{\NN}$.
Furthermore, we  define a linear map
$$\bar{\pi}_{A}:\dao\to \dao, \quad (\bar{\pi}_{A}(f))(0)=-\lambda f(0), \quad (\bar{\pi}_{A}(f))(n)=f(n-1),\, n\geq 1,\quad f\in \dao.$$
%\qhz{$$\bar{\pi}_{A}:\dao\to \dao, \quad ((\bar{\pi}_{A}(f))(n))=(f(n-1)),\, n\geq 1, \quad (\bar{\pi}_{A}(f))(0)=-\lambda f(0),\quad f\in \dao.$$}
\begin{prop}
The triple $(\dao,\bar{\partial}_{A},\bar{\pi}_{A})$ is a quasi-idempotent differential Rota-Baxter algebra of weight $\lambda$.
\end{prop}
\begin{proof}
For all $f\in \dao$, we have
\begin{equation}
\bar{\pi}_A(f)=(-\lambda f(0),f(0),-\lambda^{-1}f(0),\cdots,(-\lambda^{-1})^{n-1}f(0),\cdots)=-\lambda f,
\mlabel{eq:pif}
\end{equation}
 proving $\bar{\pi}_A(f)$ is in $\dao$.
Then
$$\bar{\pi}^2_A(f)=(\lambda^2f(0),-\lambda f(0),f(0),-\lambda^{-1}f(0),\cdots,(-\lambda^{-1})^{n-2}f(0),\cdots)=\lambda^2 f,$$
Thus $\bar{\pi}^2_A(f)=-\lambda\bar{\pi}_A(f)$, proving that $\bar{\pi}_A$ is a quasi-idempotent operator of weight $\lambda$. In addition, we have
$$((\bar{\partial}_A \circ \bar{\pi}_{A})(f))(0)=(\bar{\partial}_A(-\lambda f))(0)=(-\lambda f)(1)=f(0),$$
and
$$((\bar{\partial}_A \circ \bar{\pi}_{A})(f))(n)=(\bar{\partial}_{A}(\bar{\pi}_{A}(f)))(n)=(\bar{\partial}_A(f))(n-1)=f(n), \,n\geq 1.$$
This gives $\bar{\partial}_A \circ \bar{\pi}_{A}=\id_{\dao}$.  Then it remains to verify that $\bar{\pi}_{A}$ is a Rota-Baxter operator of weight $\lambda$ on $\dao$.
Define
\begin{eqnarray*}
h_0:=\bar{\pi}_{A}(f)\bar{\pi}_{A}(g)-\bar{\pi}_{A}(\bar{\pi}_{A}(f) g)-\bar{\pi}_{A}(f\bar{\pi}_{A}(g))-\lambda \bar{\pi}_{A}(fg).
\end{eqnarray*}
Then
\begin{eqnarray*}
\bar{\partial}_A(h_0)&=&\bar{\partial}_A\Big(\bar{\pi}_{A}(f)\bar{\pi}_{A}(g)-\bar{\pi}_{A}(\bar{\pi}_{A}(f)g)-\bar{\pi}_{A}(f\bar{\pi}_{A}(g))-\lambda \bar{\pi}_{A}(fg)\Big)\\
&=&\bar{\partial}_A\Big(\bar{\pi}_{A}(f)\bar{\pi}_{A}(g)\Big)-\bar{\partial}_A\Big(\bar{\pi}_{A}(\bar{\pi}_{A}(f)g)\Big)
-\bar{\partial}_A\Big(\bar{\pi}_A(f\bar{\pi}_{A}(g))\Big)-\bar{\partial}_A\Big(\lambda \bar{\pi}_{A}(fg)\Big)\\
&=&\bar{\partial}_A(\bar{\pi}_{A}(f))\bar{\pi}_{A}(g)+\bar{\pi}_{A}(f)\bar{\partial}_A(\bar{\pi}_{A}(g))+\lambda \bar{\partial}_A(\bar{\pi}_{A}(f))\bar{\partial}_A(\bar{\pi}_{A}(g))-\bar{\partial}_A\Big(\bar{\pi}_{A}(\bar{\pi}_{A}(f)g)\Big)\\
&&-\bar{\partial}_A\Big(\bar{\pi}_A(f\bar{\pi}_{A}(g))\Big)-\bar{\partial}_A\Big(\lambda \bar{\pi}_{A}(fg)\Big)\quad\text{(by $\bar{\partial}_A$ being a differential operator of weight $\lambda$})\\
&=&f\bar{\pi}_{A}(g)+\bar{\pi}_{A}(f)g+\lambda fg-\bar{\pi}_{A}(f)g-f\bar{\pi}_{A}(g)-\lambda fg\quad\text{(by $\bar{\partial}_A\circ \bar{\pi}_A=\id_{\dao}$)}\\
&=&0.
\end{eqnarray*}
Thus $h_0=(h_0(0),0,\cdots 0,\cdots).$
By Eq.~(\mref{eq:pif}), we get
\begin{eqnarray*}
h_0(0)&=&\bar{\pi}_{A}(f)\bar{\pi}_{A}(g)(0)-\bar{\pi}_{A}(\bar{\pi}_{A}(f) g)(0)-\bar{\pi}_{A}(f\bar{\pi}_{A}(g))(0)-\lambda \bar{\pi}_{A}(fg)(0)\\
&=&(-\lambda)f(0)(-\lambda)g(0)-\Big(-\lambda\Big(-\lambda f(0)g(0)\Big)\Big)-\Big(-\lambda\Big(-\lambda f(0)g(0)\Big)\Big)-\lambda\Big(-\lambda \Big(f(0)g(0)\Big)\Big)\\
&=&\lambda^{2}f(0)g(0)-\lambda^{2}f(0)g(0)-\lambda^{2}f(0)g(0)+\lambda^{2}f(0)g(0)\\
&=&0.
\end{eqnarray*}
Thus $h_0=0$,  and so $\bar{\pi}_{A}$ is a Rota-Baxter operator of weight $\lambda$ on $\dao$.
\end{proof}

\begin{exam}\mlabel{ex:dl}
 Let $R$ be an algebra and let $\lambda \neq 0$. Define a linear operator
 $$d_{\lambda}:R\to R, \quad x\mapsto -\lambda^{-1} x, \quad\forall x \in R.$$
Firstly, $d(1)=-\lambda^{-1}\neq 0$. Then for all $x,y\in R$,
\begin{eqnarray*}
d^{2}_{\lambda}(x)=(-\lambda^{-1})^2x=-\lambda^{-1}d_{\lambda}(x)
\end{eqnarray*}
and
\begin{eqnarray*}
d_{\lambda}(x)y+xd_{\lambda}(y)+\lambda d_{\lambda}(x)d_{\lambda}(y)
=(-\lambda^{-1}x)y+x(-\lambda^{-1}y)+\lambda (-\lambda^{-1}x)(-\lambda^{-1}y)
=d_{\lambda}(xy).
\end{eqnarray*}
Then  $d_\lambda$ is a degenerate quasi-idempotent differential operator of weight $\lambda$.
\end{exam}
\begin{exam}\mlabel{ex:dll}
 Let $R$ be an  algebra and let $\lambda\neq 0$. Define a linear operator
 $$P_{\lambda}:R\to R, \quad x\mapsto -\lambda x, \quad\forall x \in R.$$
Let $x,y\in R$. Then
\begin{eqnarray*}
P_{\lambda}(P_{\lambda}(x)y)+P_{\lambda}(xP_{\lambda}(y))+\lambda P_{\lambda}(xy)
=\lambda^{2}xy+\lambda^{2}xy-\lambda^{2}xy
=P_{\lambda}(x)P_{\lambda}(y),
\end{eqnarray*}
and
\begin{eqnarray*}
P^{2}_{\lambda}(x)=P_{\lambda}(-\lambda x)=-\lambda P_{\lambda}(x).
\end{eqnarray*}
Then  $P_\lambda$ is a quasi-idempotent Rota-Baxter operator of weight $\lambda$. Furthermore, $(d_\lambda \circ P_\lambda)(x)=x$, where $d_\lambda$ is   given in Example~\mref{ex:dl}. Thus, $(R,d_\lambda,P_\lambda)$ is a degenerate quasi-idempotent differential Rota-Baxter algebra of
weight $\lambda$.
\end{exam}

At the end of this section,  we will see that every quasi-idempotent element of $R$ can induce a degenerate quasi-idempotent differential operator, a quasi-idempotent Rota-Baxter operators and further a quasi-idempotent differential Rota-Baxter operator.

\begin{prop}~\cite[Proposition 2.2]{Jian17}
Let $R$ be an algebra. Let $\xi\in R$ be a quasi-idempotent element of weight $\lambda$.  Then the linear operator $P_\xi: R\to R$ given by $x\mapsto \xi x$, is a quasi-idempotent Rota-Baxter operator of weight $\lambda$ on $R$.
\mlabel{prop:qro}
\end{prop}
Similarly, we have
\begin{prop}
 Let R be a commutative $\mathbf{k}$-algebra.  If $\xi \in R $ be an invertible quasi-idempotent element of weight $\lambda$, then the linear operator $d_{\xi}: R\to R$ given by $x\to \xi^{-1} x$, is a degenerate quasi-idempotent differential operator of weight $\lambda$.
\mlabel{prop:qdo}
\end{prop}
\begin{proof}
By $\xi^2=-\lambda \xi$, we get $\lambda (\xi^{-1})^2=-\xi^{-1}$ and $\lambda\neq 0$. Then for all $x,y\in R$, we have
\begin{eqnarray*}
d^{2}_{\xi}(x)=(\xi^{-1})^2x
=-\lambda^{-1}d_{\xi}(x),\quad d_{\xi}(1)=\xi^{-1}\neq 0,
\end{eqnarray*}
and
\begin{eqnarray*}
d_\xi(x)y+xd_\xi(y)+\lambda d_\xi(x)d_\xi(y)
=\xi^{-1}xy+\xi^{-1}xy+\lambda(\xi^{-1})^2xy
=d_\xi (xy).
\end{eqnarray*}
Thus $d_{\xi}$ is a degenerate quasi-idempotent differential operator of weight $\lambda$.
\end{proof}

\begin{prop}
Let R be a commutative $\mathbf{k}$-algebra.  Let $\xi\in R$ be an invertible quasi-idempotent element of weight $\lambda$. Let $d_\xi$ and $P_\xi$ be as above. Then
\begin{enumerate}
\item The pair $(R,d_{\xi})$ is a commutative degenerate quasi-idempotent differential $\mathbf{k}$-algebra of weight $\lambda$.
\mlabel{it:cdq}
\item The pair $(R,P_{\xi})$ is a commutative quasi-idempotent Rota-Baxter $\mathbf{k}$-algebra of weight $\lambda$.
\mlabel{it:cqr}
\item The triple $(R,d_{\xi},P_{\xi})$ is a commutative degenerate quasi-idempotent differential Rota-Baxter $\mathbf{k}$-algebra of weight $\lambda$.
\mlabel{it:cqdr}
\end{enumerate}
\end{prop}
\begin{proof}
(\mref{it:cdq}) follows from Proposition~\mref{prop:qdo}.
\smallskip

\noindent
(\mref{it:cqr}) follows from Proposition~\mref{prop:qro}.
\smallskip

\noindent
(\mref{it:cqdr}) For all $x\in R$, we have
\begin{eqnarray*}
(d_{\xi}\circ P_{\xi})(x)=d_{\xi}(\xi x)=\xi^{-1}(\xi x) =x,
\end{eqnarray*}
proving $d_{\xi}\circ P_{\xi}=\id_{R}$. Then by Items~(\mref{it:cdq}) and~(\mref{it:cqr}), $(R, d_\xi, P_\xi)$ is a degenerate quasi-idempotent differential Rota-Baxter algebra of weight $\lambda$.
\end{proof}

\section{Gr\"obner-Shirshov bases  for free commutative $\Omega$-operated algebras}
\mlabel{sec:foa}
In this section, we first recall from~\cite{GGR18,Qiu14}  the construction of free commutative $\Omega$-operated unitary algebras on a set $Y$, and the Composition-Diamond lemma for free commutative $\Omega$-operated algebras.
\subsection{Free commutative $\Omega$-operated algebras}
An {\bf $\Omega$-operated algebra} (or called an  {\bf algebra with multiple operators}) is defined to be an algebra $R$  equipped with a set $\Omega$ of multiple linear operators.

Let $Y$ be a set. Denote by $C(Y)$ the free commutative monoid on $Y$ with the identity $1$. Let
\begin{eqnarray*}
\Omega=\bigcup^{\infty}_{n=1} \Omega_{n},
\end{eqnarray*}
where $\Omega_{n}$ is the set of $n$-ary operators. We first construct the free commutative $\Omega$-operated monoid on $Y$ by a direct system $$\{\,\iota_{k}: \frakc_k\to \frakc_{k+1}\, \}_{k=0}^\infty$$
of free commutative monoids $\frakc_n$, where  $\iota_{n}$  is the natural embedding.

Let
\begin{eqnarray*}
Y_{0}:=Y\quad\text{and}\quad \frakc_{0}:=C(Y_{0}).
\end{eqnarray*}
Then we define
\begin{eqnarray*}
Y_{1}:=Y\sqcup \Omega(\frakc_{0})\quad \text{and}\quad\frakc_{1}:=C(Y_{1}),
\end{eqnarray*}
where
\begin{eqnarray*}
\Omega(\frakc_{0})=\bigcup^{\infty}_{n=1}\{\,\omega_n (u_{1},u_{2},\cdots ,u_{n})\,|\,\omega_n\in \Omega_{n},u_{i}\in \frakc_{0},i=1,2,\cdots,n\}.
\end{eqnarray*}
The injection $\iota_0: Y_0 \to Y_1$ induces an embedding from  $ \frakc_0$ to $\frakc_1$, still denoted by $\iota_0$.
For a given $k\geq 0$, assume by induction that we have  defined the commutative $\Omega$-operated monoid $\frakc_i$ with the properties that $\frakc_{i}=C(Y\sqcup\Omega(\frakc_{i-1}))$ and the natural embedding $\iota_{i-1}:\frakc_{i-1}\to\frakc_{i}$ for $0\leq i\leq k$.
Set
\begin{eqnarray*}
Y_{k+1}:=Y\sqcup \Omega(\frakc_{k})\quad\text{and}\quad \frakc_{k+1}:=C(Y_{k+1}),
\end{eqnarray*}
where
\begin{eqnarray*}
\Omega(\frakc_{k})=\bigcup^{\infty}_{n=1}\{\,\omega_n (u_{1},u_{2},\cdots ,u_{n})\,|\,\omega_n\in \Omega_{n},u_{i}\in \frakc_{k},i=1,2,\cdots,n\}.
\end{eqnarray*}
Then the identity map on $Y$ and $\iota_{k-1}$ together induce an injection
$$\iota_{k}:Y\sqcup\Omega(\frakc_{k-1})\to Y\sqcup\Omega(\frakc_{k}).$$
Then by the functoriality of $C$, we get an embedding, also denoted by $\iota_k$,
$$\iota_k: \frakc_k\to\frakc_{k+1}.$$
This completes the desired direct system.  Then by taking the direct limit, we obtain a commutative monoid
\begin{eqnarray*}
\frakc(Y):=\dirlim
\frakc_k=\bigcup_{k\geq 0}\frakc_{k}.
\end{eqnarray*}
Applying the direct limit on both sides of the equation
$\frakc_k=C(Y\sqcup\Omega(\frakc_{k-1}))$, we have
$$\frakc(Y)=C(Y\sqcup\Omega(\frakc(Y))).$$
It follows that $\Omega(\frakc(Y)\subseteq \frakc(Y)$.  Then every nonunit element $u\in\frakc(Y)$ has a unique  expression of the form
\begin{equation}
u=u_1\cdots u_r,
\end{equation}
where $r\geq 1$ and $u_i\in Y\sqcup \Omega(\frakc(Y)), i=1,\cdots, r$, which is called a {\bf prime $\Omega$-word}. We call $r$ the {\bf breadth} of $u$, denoted by $|u|$, with the convention that $|u|=0$ if $u=1$. The {\bf depth} $\dep(u)$ of $u$ is said to be the least $n\geq 0$
such that $u$ is contained in  $\frakc_n$. Thus $\dep(u)=0$ if and only if $u\in\frakc_0= C(Y)$.

Define a set map
$$\lc\,\rc_{\omega_n}: \frakc(Y)^{n}\to \frakc(Y), \quad (u_1,u_2,\cdots, u_n)\mapsto \omega_n(u_1,u_2,\cdots,u_n),\quad \forall\omega_n\in \Omega_n.$$
Let $\lc\,\rc_{\Omega_n}:=\{\lc\,\rc_{\omega_n}\,|\,\omega_n\in\Omega_n\}$. Let $\lc\,\rc_\Omega:=\bigcup_{n\geq 1}^{\infty}\{\lc\,\rc_{\Omega_n}\,|\,\Omega_n\in\Omega\}$. Then $\frakc(Y)$ together with $\lc\,\rc_\Omega$ forms a commutative $\Omega$-operated monoid. Denote by $\bfk\frakc(Y)$  the  $\bfk$-space with basis $\frakc(Y)$. Extending the multiplication on $\frakc(Y)$ by bilinearity and the set maps $\lc\,\rc_\Omega$ by multilinearity, also dented by $\lc\,\rc_\Omega$,  we obtain a commutative $\Omega$-operated algebra $\bfk\frakc(Y)$. An element of $\frakc(Y)$ (resp. $\bfk\frakc(Y)$) is called a {\bf commutative $\Omega$-word}(resp. {\bf commutative $\Omega$-polynomial}).
If $Y$ is a finite set, we can also just list its elements, as in $\bfk\frakc(x,y)$ when $Y = \{x, y\}$. Let $j_Y:Y\to \bfk\frakc(Y)$ be the natural embedding.
\begin{prop}~\cite{GGZ14,GGR18,Qiu14} Let $j_Y$,  $\lc\,\rc_\Omega$  and $\bfk\frakc(Y)$ be as above.
Then  $\bfk\frakc(Y):=(\bfk\frakc(Y), \lc\,\rc_\Omega, j_Y)$ is  the free commutative $\Omega$-operated unitary algebra on set $Y$.
\end{prop}
\delete{
\begin{prop}\cite{ZGG23}\mlabel{prop:quo}
Let $X = \{u_1,\cdots,u_n\}$ and let $\phi(u_1,\cdots,u_n)\in\bfk\frakc(X)$ with $n\geq 1$.  Let $\Id_\phi$ be the $\Omega$-operated ideal in $\bfk\frakc(Y)$  generated by the following set
$$\{\phi(r_1,\cdots,r_n)\,|\,r_1,r_2,\cdots,r_n\in \bfk\frakc(Y)\}.
$$
and let $\Pi_\phi: \bfk\frakc(Y)\to\bfk\frakc(Y)/\Id_\phi$ be the quotient morphism. Then the quotient $\Omega$-operated algebra $\bfk\frakc(Y)/\Id_\phi$, together
with $i_Y:=\Pi_\phi\circ j_Y$ and the operator $P_\Omega$ induced by $\lc\,\rc_\Omega$, is the free commutative $\phi$-algebra on $Y$.
\end{prop}
}
Now we set  $X = \{x_1,\cdots,x_n\}$. Let $\phi(x_1,x_2,\cdots,x_n)$ be a commutative $\Omega$-polynomial in $\bfk\frakc(X)$.  Let $(R,P_\Omega)$ be a commutative $\Omega$-operated algebra, where $P_\Omega:=\{P_{\omega_n}:R^n\to R\,|\,\omega_n\in\Omega_n,n\geq 1\}$, are the multiple linear operators on $R$. Define a set map by defining
$$f:X\to R, \quad x_i\mapsto r_i,\quad i=1,\cdots,n.$$
By the universal property of $\bfk\frakc(X)$, there exists a unique $\Omega$-operated algebra morphism $\bar{f}:\bfk\frakc(X)\to R$.  Then we denote
$$\phi_R(r_1,r_2,\cdots,r_n):=\bar{f}\Big(\phi(x_1,\cdots,x_n)\Big).$$
\begin{defn}~\cite{QQWZ21,ZGG23} Let $X=\{x_1,\cdots,x_n\}$ and let $\phi(x_1,\cdots,x_n)\in\bfk\frakc(X)$.
\begin{enumerate}
\item We say that a commutative $\Omega$-operated algebra $(R,P_\Omega)$ is a {\bf commutative $\phi$-algebra}, if
$$\phi_R(r_1,\cdots, r_n)=0\quad \text{for all}\, r_1,\cdots, r_n\in R.$$
In this case, $P_\Omega$ is called a {\bf $\phi$-operator}.
\item We call $\phi(x_1,\cdots, x_n)=0$ (or simply $\phi(x_1,\cdots, x_n)$) a {\bf commutative $\Omega$-operated polynomial identity} (short for {\bf $\Omega$-COPI}).
\item Let $\Phi\subset \bfk\frakc(X)$ be a  family of $\Omega$-COPIs.  A commutative $\Omega$-operated algebra $(R, P_\Omega)$ is called a {\bf commutative $\Phi$-algebra} if it is a $\phi$-algebra for all $\phi\in \Phi$.
\item Let $S$ be a subset of $R$. The smallest operated ideal of $R$ containing $S$ is called the {\bf $\Omega$-operated ideal generated by $S$}, denoted by $\Id(S)$.
%\item A commutative $\phi$-algebra is called a {\bf commutative $\Omega$-operated polynomial identity algebra} (short for {\bf COPI-algebra}).
\end{enumerate}
\end{defn}
\begin{prop}\cite{GGR18,QQWZ21}\mlabel{prop:quos} Let $\Phi\subset \bfk\frakc(X)$ be a  family of $\Omega$-COPIs. Let $R=\bfk\frakc(Z)$ be the free commutative $\Omega$-operated algebra on the set $Z$ with the natural embedding $j_Z: Z\to \bfk\frakc(Z)$. Let $\Id_\Phi(Z)$ be the $\Omega$-operated ideal in $R$ generated by the following set
$$\{\phi_R(r_1,\cdots,r_n)\,|\, r_1,\cdots, r_n\in R,\,\phi(u_1,\cdots u_n)\in \Phi\}.$$
Let $\Pi_{\Phi}: R\to R/\Id_\Phi(Z)$ be the quotient morphism. Then the quotient $\Omega$-operated algebra $R/\Id_\Phi(Z)$, together
with $i_Z:=\Pi_{\Phi}\circ j_Z$ and the operator $P_\Omega$ induced by $\lc\,\rc_\Omega$, is the free commutative $\Phi$-algebra on $Z$.
\end{prop}
From this, we obtain
\begin{prop}\mlabel{prop:drb}Let $X=\{x,y\}$ and let $\Omega=\{\lc\, \rc_{\rm d},\,\lc\, \rc_{\rm p}\}$ be a set of two distinct linear operators on $\bfk\frakc(x,y)$. Let $R=\bfk\frakc(Z)$ be the free commutative $\Omega$-operated algebra on the set $Z$ with the natural embedding $j_Z: Z\to \bfk\frakc(Z)$. Let $0\neq \lambda\in \bfk$.
Denote
 $$\Phi_{\rm d}:=\left\{\begin{array}{lll}
 \phi_1(x,y):=\lc x\rc _{\rm d}\lc y\rc _{\rm d}+\lambda^{-1}\lc x\rc _{\rm d}y+\lambda^{-1}x\lc y\rc _{\rm d}-\lambda^{-1}\lc x y\rc _{\rm d},\\
 \phi_2(x,y):=\lc x\rc^2 _{\rm d}+\lambda^{-1} \lc x\rc _{\rm d}\\
 \end{array}\right\}\subset \bfk\frakc(x,y).$$
 $$\Phi_{\rm p}:=\left\{\begin{array}{lll}
 \phi_3(x,y):=\lc x\rc_{\rm p} \lc y\rc_{\rm p}-\lc x\lc y\rc_{\rm p}\rc_{\rm p}-\lc \lc x\rc_{\rm p}y\rc_{\rm p}-\lambda \lc xy\rc_{\rm p},\\
 \phi_4(x,y):=\lc x\rc^2_{\rm p}+\lambda \lc x\rc_{\rm p}\\
 \end{array}\right\}\subset \bfk\frakc(x,y).$$
 $$\Phi_{\rm dp}:=\Phi_{\rm d}\cup \Phi_{\rm p}\cup\left\{\begin{array}{lll}
 \phi_5(x,y):=\lc \lc x\rc_{\rm p}\rc_{\rm d}-x
 \end{array}\right\}\subset \bfk\frakc(x,y).$$
 \begin{enumerate}
 \item Let $\Pi_{\Phi_{\rm d}}: R\to R/\Id_{\Phi_{\rm d}}(Z)$ be the quotient morphism. Then the quotient $\Omega$-operated algebra $R/\Id_{\Phi_{\rm d}}(Z)$, together
with $i_Z:=\Pi_{\Phi_{\rm d}}\circ j_Z$ and the operator $d$ induced by $\lc\,\rc_{\rm d}$, is the free commutative quasi-idempotent differential algebra on $Z$. \mlabel{it:fqd}
 \item Let $\Pi_{\Phi_{\rm p}}: R\to R/\Id_{\Phi_{\rm p}}(Z)$ be the quotient morphism. Then the quotient $\Omega$-operated algebra $R/\Id_{\Phi_{\rm p}}(Z)$, together
with $i_Z:=\Pi_{\Phi_{\rm p}}\circ j_Z$ and the operator $P$ induced by $\lc\,\rc_{\rm p}$, is the free commutative quasi-idempotent Rota-Baxter algebra on $Z$. \mlabel{it:fqrb}
 \item Let $\Pi_{\Phi_{\rm dp}}: R\to R/\Id_{\Phi_{\rm dp}}(Z)$ be the quotient morphism. Then the quotient $\Omega$-operated algebra $R/\Id_{\Phi_{\rm dp}}(Z)$, together
with $i_Z:=\Pi_{\Phi_{\rm dp}}\circ j_Z$ and the operators $d$  and $P$  induced by $\lc\,\rc_{\rm d}$  and $\lc \rc_{\rm p}$ respectively, is the free commutative quasi-idempotent differential Rota-Baxter algebra on $Z$. \mlabel{it:fqdrb}
 \end{enumerate}
\end{prop}

\subsection{Composition-Diamond lemma}

Let $Y$ be a set and let $\star$ be a symbol not in $Y$. Let $Y^{\star}:=Y\cup\{\star\}$. By a {\bf $\star$-$\Omega$-word}, we mean that any word in $\frakc(Y^\star)$ with exactly one occurrence of $\star$. The set of all $\star$-$\Omega$-words on $Y$ is denoted by $\frakc^\star(Y).$

Let $q\in \frakc^\star(Y)$ and let $s\in \bfk\frakc(Y)$.  We will use the notation $q|_s$ or $q|_{\star\mapsto s}$ to denote the $\Omega$-word on $Y$ obtained by replacing $\star$ in $q$ by $s$.

We now suppose that $\frakc(Y)$ is equipped with a monomial order $\leq$.  This means that $\leq$ is well order on $\frakc(Y)$ such that for any $u, v\in \frakc(Y)$ and any $q\in \frakc^\star(Y)$,
$$u \leq v\Rightarrow q|_u \leq q|_v.$$

For each commutative $\Omega$-polynomial $f\in  \bfk\frakc(Y)$, we use the notation  $\bar{f}$ to denote  the leading monomial of $f$ with respect to the monomial order $\leq$.  We call $f$ {\bf monic} if the coefficient of $\bar{f}$ is 1.
\begin{defn}Let $\leq$ be a monomial order on $\frakc(Y)$.  Let $ f , g \in\bfk\frakc(Y)$ be distinct monic $\Omega$-polynomials.
\begin{enumerate}
\item
If there exist $\omega, \mu, \nu\in \frakc(Y)$ such that $\omega=\bar{f}\mu=\nu\bar{g}$ with $\max\{\,|\bar{f}|,\,|\bar{g}|\,\}<|\omega| < |\bar{f}|+ |\bar{g}|$,  we call  $$(f , g)^{\mu,\nu}_\omega:= f\mu-\nu g $$ the {\bf intersection composition of  $f$ and $g$ with respect to $(\mu,\nu)$}.
\item
If there exist $q\in \frakc^{\star}(Y)$ and $\omega \in \frakc(Y)$ such that $\omega = \bar{f}= q|_{\bar{g}}$ , we call
$$(f , g)^{q}_\omega: = f - q|_g$$ the {\bf including composition of $f$ and $g$  with respect to $q$}.
\end{enumerate}
\end{defn}
The $\Omega$-word $\omega$ presented in the above definition is called the {\bf ambiguity} of $(f , g)^{\mu,\nu}_\omega$ or $(f , g)^{q}_\omega$.
\delete{Denote by $\Id(f, g) $ is the ideal of $\bfk\frakc(Y)$ generated by $f$ and $g$. Then
\begin{eqnarray*}
(f, g)^{\mu,\nu}_w, (f,g)^{q}_w\in \Id(f, g) \quad \text{and} \quad  \overline{(f, g)^{\mu,\nu}_w},\, \overline{(f,g)^{q}_w}<w,
\end{eqnarray*}
}
\begin{defn}
Let $\leq$ be a monomial order on $\frakc(Y)$.  Let $S$ be a set of monic commutative $\Omega$-polynomials in $\bfk\frakc(Y)$. Let $\omega\in\frakc(Y)$.
\begin{enumerate}
 \item A  commutative $\Omega$-polynomial $f$ is called {\bf trivial modulo\,$(S, \omega)$} if
\begin{eqnarray*}
f=\sum_i c_{i} q_{i}|_{s_{i}},
\end{eqnarray*}
where $c_{i}\in \bfk$, $q_{i}\in \frakc^{\star}(Y)$, $s_{i}\in S$ and $q_{i}|_{\bar{s_{i}}}<\omega$,   and we denote it by
 \begin{eqnarray*}
f\equiv 0 \mod(S, \omega).
 \end{eqnarray*}
\item
For any commutative $\Omega$-polynomials $f$ and $g$, a pair $(f,g)$ is called {\bf congruent modulo $(S,\omega)$}, denoted by $f\equiv g\mod(S, \omega)$, if $f-g$ is trivial modulo $(S,\omega)$.
\item
The set $S$ is called a {\bf Gr\"obner-Shirshov basis with respective to $\leq$}  if for all $f,g\in S$, both
$(f, g)^{\mu,\nu}_\omega$ and $(f,g)^{u}_\omega$ are trivial modulo\, $(S,\omega)$.
\end{enumerate}
\end{defn}

\begin{theorem}$(${\rm Composition-Diamond Lemma}$)$\cite[Thoerem 3.2]{Qiu14}\mlabel{thm:cdl}
  Let $S \subseteq\bfk\frakc(Y)$  be a set of monic commutative $\Omega$-polynomials. Let $\leq$ be a monomial order on $\frakc(Y)$. Then the following statements are equivalent:
  \begin{enumerate}
    \item $S$ is a Gr\"obner-Shirshov basis.
    \item As $\bfk$-vector spaces, $$\bfk\frakc(Y)=\bfk\Irr(S)\oplus \Id(S),$$
where $\Irr(S):=\frakc(Y)\backslash\{\,q|_{\bar{s}}\,|\,q\in\frakc^\ast(Y),s\in S\}$, and so $\Irr(S)$ is a $\bfk$-basis of $\bfk\frakc(Y)/\Id(S)$. \mlabel{it:gsb}
  \end{enumerate}

\end{theorem}
\subsection{A monomial order on $\frakc(Y)$}
In this section, we will give a monomial order on $\frakc(Y)$  when the well-ordered set $\Omega=\Omega_1=\{\lc\, \rc_{\rm d},\lc\, \rc_{\rm p}\}$  is composed of two distinct linear operators on $\bfk\frakc(Y)$ such that  $\lc\, \rc_{\rm d}>\lc\, \rc_{\rm p}$. We define ${\rm d}> {\rm p}$ if and only if $\lc\, \rc_{\rm d}>\lc\, \rc_{\rm p}$.
Note that an arbitrary elements of $\frakc(Y)$ has a unique expression, that is
\begin{equation}\mlabel{eq:unidec}
u=u_{0}\lc u^{\ast}_{1}\rc_{\omega_1} u_{1}\lc u^{\ast}_{2}\rc_{\omega_2} \cdots \lc u^{\ast}_{r}\rc_{\omega_r} u_{r},
\end{equation}
where $u_{0},u_{1},\cdots,u_{r}\in C(Y)$, $\lc u^{\ast}_{1}\rc_{\omega_1},\lc u^{\ast}_{2}\rc_{\omega_2},\cdots,\lc u^{\ast}_{r}\rc_{\omega_r}\in \Omega(\frakc(Y))$ and $\omega_i\in\{{\rm d},{\rm p}\}$ for $i=1,\cdots, r$. Then
\begin{equation}\mlabel{eq:breo}
|u|_{\Omega} :=r
\end{equation}
is called the {\bf $\Omega$-breath}  of $u$.
\begin{defn}
Let $u,v\in \frakc(Y)$.
  \begin{enumerate}
    \item Define
    \begin{eqnarray*}
    u\leq_{\dgo} v \,\Leftrightarrow\, \deg_{\Omega}(u)\leq \deg_{\Omega}(v),
    \end{eqnarray*}
    where $\deg_{\Omega}(u)$ is the number of occurrence of $\omega\in\Omega$ in $u$.
    \item Define
    \begin{eqnarray*}
    u\leq_{\bro}v\,\Leftrightarrow\, |u|_{\Omega}\leq |v|_{\Omega},
    \end{eqnarray*}
    where $|u|_{\Omega}$ (resp. $|v|_{\Omega}$) is  the $\Omega$-breadth of $u$ (resp. $v$) defined in Eq.~(\mref{eq:breo}).
  \end{enumerate}
\end{defn}

 We next construct a monomial order $\leq_{\db}$ on $\frakc(Y)$. Since $\frakc(Y)=\bigcup_{k\geq 0} \frakc_k$, we do this by induction on $k\geq 0$. Let $(Y,\leq)$ be a well-ordered set.
\begin{enumerate}
  \item
 Let $u, v \in \frakc_{0}(=C(Y))$.  Write
 $$u=u_1\cdots u_r\quad\text{and} \quad v=v_1\cdots v_s,$$
 where $u_1,\cdots, u_r, v_1,\cdots, v_s\in Y$.
  Then define
  \begin{eqnarray*}
  u\leq_0 v\, \Leftrightarrow\, u\leq_{\dlex}v,
  \end{eqnarray*}
  where $\leq_{\dlex}$ is the degree lexicographical order, that is,
 $$u\leq_{\dlex}v\, \Leftrightarrow\, (\deg_Y(u),u_1,\cdots,u_r)\leq (\deg_Y(v),v_1,\cdots, v_s).$$
Here  $\deg_{Y}(u)$ is the number of occurrence of $y\in Y$ in $u$.
  \item
For a given $k\geq 1$,  suppose that  a well order $\leq_k$ on $\frakc_k$ has been defined.   Let  $u, v \in \frakc_{k+1}$. Since
  $\frakc_{k+1}=C(Y\sqcup \Omega(\frakc_{k}))$, we can  write
$$
u=u_{0}\lfloor u^{\ast}_{1}\rfloor_{\alpha_1} u_{1}\lfloor u^{\ast}_{2}\rfloor_{\alpha_2} \cdots \lfloor u^{\ast}_{r}\rfloor_{\alpha_r} u_{r}\quad\text{and}\quad
v=v_{0}\lfloor v^{\ast}_{1}\rfloor_{\beta_1} v_{1}\lfloor v^{\ast}_{2}\rfloor_{\beta_2} \cdots \lfloor v^{\ast}_{s}\rfloor_{\beta_s} v_{s},
$$
where $u_{0},v_{0},u_{i},v_{j}\in C(Y)$, $\alpha_i,\beta_i\in \{{\rm d}, {\rm p}\}$ and $u^{\ast}_{i}, v^{\ast}_{j}\in \frakc_k$ for $i=1,\cdots,r$ and $j=1,\cdots,s$. We first assume that $r=s$. Then define
  \begin{eqnarray*}
  u\leq_{{\rm lex}_{k+1}}v \Leftrightarrow &(\alpha_1,\alpha_2,\cdots,\alpha_r,u^{\ast}_{1},u^{\ast}_{2},\cdots,u^{\ast}_{r},u_{0},u_{1},\cdots,u_{r})\\
  &\leq
  (\beta_1,\beta_2,\cdots,\beta_r,v^{\ast}_{1},v^{\ast}_{2},\cdots,v^{\ast}_{r},v_{0},v_{1},\cdots,v_{r}).
  \end{eqnarray*}
 Then by the induction hypothesis and \cite[Lemma~5.4.(b)]{ZGGW}, $\leq_{{\rm lex}_{k+1}}$ is a well order.  If $\Omega$ is a singleton, then $\alpha_i=\beta_i$ for $1\leq i\leq r$. So
 $\leq_{{\rm lex}_{k+1}}$ is the same as  the case of the order $\leq_{\rm db}$ defined in \cite[Lemma~5.5]{ZGGW}. This is the only difference between $\leq_\db$ and $\leq_{\rm db}$.

 Now define
  $$u\leq_{k+1}v\, \Leftrightarrow\,
  \left \{\begin{array}{ll} u<_{\dgo}v, \\
  \text{or}\;u=_{\dgo}v \;\text{and}\; u<_{\bro}v,  \\
  \text{or}\;u=_{\dgo}v,\; u=_{\bro}v(=r) \;\text{and}\; u<_{{\rm lex}_{k+1}}v.\\
  \end{array} \right .$$
\end{enumerate}
Then by \cite[Lemma~5.4.(a)]{ZGGW}, $\leq_{k+1}$ is a well order on $\frakc_{k+1}$.
Finally define the order
\begin{equation}\mlabel{eq:mord}
\leq_{\db}:=\bigcup_{k\geq 0}\leq_k.
\end{equation}
\begin{prop}The  order $\leq_{\db}$ is a monomial order on $\frakc(Y)$.
\end{prop}
\begin{proof}
It follows from the same proof of \cite[Theorem~5.8]{ZGGW}.
\end{proof}

\section{Free commutative quasi-idempotent differential Rota-Baxter algebras}
\mlabel{sec:fcqa}
In this section  we will give a linear basis of free  commutative quasi-idempotent differential algebras (resp. Rota-Baxter algebras, resp. differential Rota-Baxter algebras) by using the Composition-Diamond lemma.
\subsection{Linear bases of free commutative quasi-idempotent differential algebras}
Firstly, this section will give a Gr\"obner-Shirshov basis for  free commutative quasi-idempotent differential algebras, leading to a linear basis of free commutative quasi-idempotent differential algebras. For this, $\Omega$ is taken to be the singleton $\{\lc\,\rc_{\rm d}\}$. We present our main result of this section as below.
\begin{theorem}\mlabel{thm:fcqd}Let $Y$ be a well-ordered set. Let $\lambda \neq 0$ and let
$$S_{\rm d}:=\left \{\begin{array}{lll}\,\phi_1(u,v):=\lc u\rc_{\rm d}\lc v\rc_{\rm d}+\lambda^{-1}\lc u\rc_{\rm d}v+\lambda^{-1}u\lc v\rc_{\rm d}-\lambda^{-1}\lc uv\rc_{\rm d}\\
 \phi_2(u,v):=\lc u \rc^2_{\rm d}+\lambda^{-1} \lc u\rc_{\rm d}
 \end{array}\bigg|\,u,v\in \frakc(Y)\right\}.
$$
With the monomial order $\leq_{\db}$ defined in Eq.~(\mref{eq:mord}),
\begin{enumerate}
\item $S_{\rm d}$ is a Gr\"obner-Shirshov basis in $\bfk\frakc(Y)$.\mlabel{it:gsd}
\item $$\Irr(S_{\rm d}):=\frakc(Y)\,\Big\backslash\,\Big\{q_1|_{\lc u\rc_{\rm d}\lc v\rc_{\rm d}}, q_2|_{\lc u\rc^2_{\rm d}}|\,q_1,q_2\in\frakc^\star(Y),u,v\in \frakc(Y)\Big\}$$
is a linear basis of the free commutative quasi-idempotent differential algebra on $Y$.\mlabel{it:fd}
\end{enumerate}
\mlabel{thm:d}
\end{theorem}
\begin{proof}
(\mref{it:gsd})
  Denote by $i\wedge j$ the composition of $\Omega$-polynomials of types $\phi_i(u,v)$ and $\phi_j(u,v)$ for $i,j=1,2$. The ambiguities of all possible compositions of commutative $\Omega$-polynomials in $S$ are presented as below.
 $$
\begin{tabular}{|c|c|c|c|c|c|}
\hline
\diagbox{i}{$i\wedge j$}{j}&1&2\\
\hline
1&\makecell{$\lc u\rc_{\rm d}\lc v\rc_{\rm d}\lc w\rc_{\rm d}$\\
$\lc q|_{\lc u\rc_{\rm d}\lc v\rc_{\rm d}}\rc_{\rm d}\lc w\rc_{\rm d}$}&$\lc q|_{\lc u\rc^{2}_{\rm d}}\rc_{\rm d}\lc v\rc_{\rm d}$\\
\hline
2&$\lc  q|_{\lc u\rc_{\rm d}\lc v\rc_{\rm d}}\rc^2_{\rm d}$&$\lc q|_{\lc u\rc^2_{\rm d}}\rc^2_{\rm d}$\\
\hline
\end{tabular}
$$
We only verify the first case: $1\wedge 1$. The other cases are easy to check.
 Let
\begin{eqnarray*}
f:=\phi_1(u,v)=\lc u\rc_{\rm d}\lc v\rc_{\rm d}+\lambda^{-1}\lc u\rc_{\rm d}v+\lambda^{-1}u\lc v\rc_{\rm d}-\lambda^{-1}\lc  uv\rc_{\rm d},\\
g:=\phi_1(z,w)=\lc z\rc_{\rm d}\lc w\rc_{\rm d}+\lambda^{-1}\lc z\rc_{\rm d}w+\lambda^{-1}z\lc w\rc_{\rm d}-\lambda^{-1}\lc  zw\rc_{\rm d}.
\end{eqnarray*}
By the monomial order $\leq_{\db}$, we obtain $\bar{f}=\lc u\rc_{\rm d}\lc v\rc_{\rm d}$ and $\bar{g}=\lc z\rc_{\rm d}\lc w\rc_{\rm d}$.
\noindent

\smallskip
{\bf(The case of intersection compositions)}. Suppose that $\omega=\bar{f}\mu= \nu\bar{g}$  with $\max\{\,|\bar{f}|,\,|\bar{g}|\,\}<|\omega| < |\bar{f}|+ |\bar{g}|$. Then $\mu=\lc w\rc_{\rm d}$ and $\nu=\lc u\rc_{\rm d}$. Thus $v=z$, and $\omega=\lc u\rc_{\rm d}\lc v\rc_{\rm d}\lc w\rc_{\rm d}$.
{\allowdisplaybreaks
\begin{eqnarray*}
(f,g)_{\omega}^{\mu,\nu}&=&f\mu-\nu g\\
&=&\Big(\lc u\rc_{\rm d}\lc v\rc_{\rm d}+\lambda^{-1}\lc u\rc_{\rm d}v+\lambda^{-1}u\lc v\rc_{\rm d}-\lambda^{-1}\lc  uv\rc_{\rm d}\Big)\lc w\rc_{\rm d}\\
&&-\lc u\rc_{\rm d}\Big(\lc v\rc_{\rm d}\lc w\rc_{\rm d}+\lambda^{-1}\lc v\rc_{\rm d}w+\lambda^{-1}v\lc w\rc_{\rm d}-\lambda^{-1}\lc  vw\rc_{\rm d}\Big)\\
&=&\lambda^{-1}u\lc v\rc_{\rm d}\lc w\rc_{\rm d}-\lambda^{-1}\lc  uv\rc_{\rm d}\lc w\rc_{\rm d}-\lambda^{-1}\lc u\rc_{\rm d}\lc v\rc_{\rm d}w+\lambda^{-1}\lc u\rc_{\rm d}\lc  vw\rc_{\rm d}\\
&\equiv&\lambda^{-1}u\Big(-\lambda^{-1}\lc v\rc_{\rm d}w-\lambda^{-1}v\lc w\rc_{\rm d}+\lambda^{-1}\lc  vw\rc_{\rm d}\Big)\\
&&-\lambda^{-1}\Big(-\lambda^{-1}\lc  uv\rc_{\rm d}w-\lambda^{-1}uv\lc w\rc_{\rm d}+\lambda^{-1}\lc  uvw\rc_{\rm d}\Big)\\
&&-\lambda^{-1}\Big(-\lambda^{-1}\lc u\rc_{\rm d}v-\lambda^{-1}u\lc v\rc_{\rm d}+\lambda^{-1}\lc  uv\rc_{\rm d}\Big)w\\
&&+\lambda^{-1}\Big(-\lambda^{-1}\lc u\rc_{\rm d}vw-\lambda^{-1}u\lc  vw\rc_{\rm d}+\lambda^{-1}\lc  uvw\rc_{\rm d}\Big)\\
&\equiv&-\lambda^{-2}u\lc v\rc_{\rm d}w-\lambda^{-2}uv\lc w\rc_{\rm d}+\lambda^{-2}u\lc  vw\rc_{\rm d}\\
&&+\lambda^{-2}\lc  uv\rc_{\rm d}w+\lambda^{-2}uv\lc w\rc_{\rm d}-\lambda^{-2}\lc uvw\rc_{\rm d}\\
&&+\lambda^{-2}\lc u\rc_{\rm d}vw+\lambda^{-2}u\lc v\rc_{\rm d}w-\lambda^{-2}\lc  uv\rc_{\rm d}w\\
&&-\lambda^{-2}\lc u\rc_{\rm d}vw-\lambda^{-2}u\lc  vw\rc_{\rm d}+\lambda^{-2}\lc  uvw\rc_{\rm d}\\
&\equiv&0\quad\mod(S_{\rm d},\omega).
\end{eqnarray*}
}
Thus $(f,g)_{\omega}^{\mu,\nu}$ is trivial modulo $(S_{\rm d},\omega)$.
\noindent

\smallskip
{\bf(The case of including compositions)}. Suppose that $\omega=\bar{f}= q'|_{\bar{g}}$. In order to keep consistent with notations of variables in the above table, we set $f:=\phi_1(z,w)$ and $g:=\phi_1(u,v)$.  This gives $\lc z\rc_{\rm d}\lc w\rc_{\rm d}=q'|_{\lc u\rc_{\rm d}\lc v\rc_{\rm d}}$. If $q'=\star$, then $\lc z\rc_{\rm d}=\lc u\rc_{\rm d}$ and $\lc w\rc_{\rm d}=\lc v\rc_{\rm d}$. This gives $(f,g)_{\omega}^{q'}=0$,  which  is trivial modulo $(S_{\rm d},\omega)$. Now let $q'\neq \star$. Then $q'=\lc q\rc_{\rm d}\lc w\rc_{\rm d}$ or $q'=\lc z\rc_{\rm d}\lc q\rc_{\rm d}$, where $q$ is in $\frakc^\star(Y)$. Thus $z=q|_{\lc u\rc_{\rm d}\lc v\rc_{\rm d}}$  or $w=q|_{\lc u\rc_{\rm d}\lc v\rc_{\rm d}}$.  Since the multiplication of $\bfk\frakc(Y)$ is commutative, we only consider the first case when $z=q|_{\lc u\rc_{\rm d}\lc v\rc_{\rm d}}$. Thus $\omega=\lc q|_{\lc u\rc_{\rm d}\lc v\rc_{\rm d}}\rc_{\rm d}\lc w\rc_{\rm d}$ and $q'=\lc q\rc_{\rm d}\lc w\rc_{\rm d}$.
{\allowdisplaybreaks
\begin{eqnarray*}
(f,g)_{\omega}^{q'}&=&f-q'|_g\\
&=&\lc q|_{\lc u\rc_{\rm d}\lc v\rc_{\rm d}}\rc_{\rm d}\lc w\rc_{\rm d}+\lambda^{-1}\lc q|_{\lc u\rc_{\rm d}\lc v\rc_{\rm d}}\rc_{\rm d}w
+\lambda^{-1}q|_{\lc u\rc_{\rm d}\lc v\rc_{\rm d}}\lc w\rc_{\rm d}\\
&&-\lambda^{-1}\lc q|_{\lc u\rc_{\rm d}\lc v\rc_{\rm d}}w\rc_{\rm d}-\lc q|_{\lc u\rc_{\rm d}\lc v\rc_{\rm d}+\lambda^{-1}\lc u\rc_{\rm d}v+\lambda^{-1}u\lc v\rc_{\rm d}-\lambda^{-1}\lc uv\rc_{\rm d}}\rc_{\rm d}\lc w\rc_{\rm d}\\
&=&\lambda^{-1}\lc q|_{\lc u\rc_{\rm d}\lc v\rc_{\rm d}}\rc_{\rm d}w+\lambda^{-1}q|_{\lc u\rc_{\rm d}\lc v\rc_{\rm d}}\lc w\rc_{\rm d}-\lambda^{-1}\lc q|_{\lc u\rc_{\rm d}\lc v\rc_{\rm d}}w\rc_{\rm d}\\
&&-\lambda^{-1}\lc q|_{\lc u\rc_{\rm d}v}\rc_{\rm d}\lc w\rc_{\rm d}-\lambda^{-1}\lc q|_{u\lc v\rc_{\rm d}}\rc_{\rm d}\lc w\rc_{\rm d}+\lambda^{-1}\lc q|_{\lc uv\rc_{\rm d}}\rc_{\rm d}\lc w\rc_{\rm d}\\
&\equiv&-\lambda^{-2}\lc q|_{\lc u\rc_{\rm d}v}\rc_{\rm d}w-\lambda^{-2}\lc q|_{u\lc v\rc_{\rm d}}\rc_{\rm d}w+\lambda^{-2}\lc q|_{\lc uv\rc_{\rm d}}\rc_{\rm d}w\\
&&-\lambda^{-2}q|_{\lc u\rc_{\rm d}v}\lc w\rc_{\rm d}-\lambda^{-2}q|_{u\lc v\rc_{\rm d}}\lc w\rc_{\rm d}+\lambda^{-2}q|_{\lc uv\rc_{\rm d}}\lc w\rc_{\rm d}\\
&&+\lambda^{-2}\lc q|_{\lc u\rc_{\rm d}v}w\rc_{\rm d}+\lambda^{-2}\lc q|_{u\lc v\rc_{\rm d}}w\rc_{\rm d}-\lambda^{-2}\lc q|_{\lc uv\rc_{\rm d}}w\rc_{\rm d}\\
&&+\lambda^{-2}\lc q|_{\lc u\rc_{\rm d}v}\rc_{\rm d}w+\lambda^{-2}q|_{\lc u\rc_{\rm d}v}\lc w\rc_{\rm d}-\lambda^{-2}\lc q|_{\lc u\rc_{\rm d}v}w\rc_{\rm d}\\
&&+\lambda^{-2}\lc q|_{u\lc v\rc_{\rm d}}\rc_{\rm d}w+\lambda^{-2}q|_{u\lc v\rc_{\rm d}}\lc w\rc_{\rm d}-\lambda^{-2}\lc q|_{u\lc v\rc_{\rm d}}w\rc_{\rm d}\\
&&-\lambda^{-2}\lc q|_{\lc uv\rc_{\rm d}}\rc_{\rm d}w-\lambda^{-2}q|_{\lc uv\rc_{\rm d}}\lc w\rc_{\rm d}+\lambda^{-2}\lc q|_{\lc uv\rc_{\rm d}}w\rc_{\rm d}\\
&\equiv&0\quad\mod(S_{\rm d},\omega).
\end{eqnarray*}
}
Thus $(f,g)_{\omega}^{q'}$ is trivial modulo $(S_{\rm d},\omega)$.
\noindent

\smallskip
(\mref{it:fd})
By Proposition~\mref{prop:drb}~(\mref{it:fqd}), $\bfk\frakc(Y)/\Id(S_{\rm d})$ is the free quasi-idempotent differential algebra on $Y$.  Then by Item~(\mref{it:gsd}) and Theorem~\mref{thm:cdl}~(\mref{it:gsb}), $\Irr(S_{\rm d})$ is a linear basis of $\bfk\frakc(Y)/\Id(S_{\rm d})$.
\end{proof}

\subsection{Linear bases of free commutative quasi-idempotent Rota-Baxter algebras}
In order to obtain a linear basis of free commutative quasi-idempotent Rota-Baxter algebras, we will establish a Gr\"obner-Shirshov basis for such algebras. For this, $\Omega$ is taken to be the singleton $\{\lc\,\rc_{\rm p}\}$.
\begin{theorem}\mlabel{thm:fcqrb}Let $Y$ be a well-ordered set. Let
$$ S_{\rm rb}:=\left\{\begin{array}{lll}\phi_3(u,v): =\lc u\rc_{\rm p} \lc v\rc_{\rm p}-\lc u\lc v\rc_{\rm p}\rc_{\rm p}-\lc \lc u\rc_{\rm p}v\rc_\rp-\lambda \lc uv\rc_\rp\\
 \phi_4(u,v): =\lc u\rc^2_\rp+\lambda \lc u\rc_\rp\end{array}\bigg|\,u,v\in \frakc(Y)\right\}.
$$
With the monomial order $\leq_{\db}$ defined in Eq.~(\mref{eq:mord}),
\begin{enumerate}
\item $S_{\rm rb}$ is a Gr\"obner-Shirshov basis in $\bfk\frakc(Y)$.
\mlabel{it:gsrb}
\item $$\Irr(S_{\rm rb}):=\frakc(Y)\,\Big\backslash\, \Big\{q_1|_{\lc u\rc_{\rm p}\lc v\rc_{\rm p}}, q_2|_{\lc u\rc^2_\rp}|\,q_1,q_2\in\frakc^\star(Y),u,v\in \frakc(Y)\Big\}$$
is a linear basis of the free  commutative quasi-idempotent Rota-Baxter algebra on $Y$.
\mlabel{it:frb}
\end{enumerate}
\mlabel{thm:rb}
\end{theorem}
\begin{proof}
(\mref{it:gsrb})
We also  denote by $i\wedge j$ the composition of $\Omega$-polynomials of types $\phi_i(u,v)$ and $\phi_j(u,v)$ for $i,j=3,4$. The ambiguities of all possible compositions of commutative $\Omega$-polynomials in $S$ are presented as below.
$$
\begin{tabular}{|c|c|c|}
\hline
\diagbox{i}{$i\wedge j$}{j}&3&4\\
\hline
3& \makecell{$\lc u\rc_{\rm p}\lc v\rc_{\rm p}\lc w\rc_\rp$\\
$\lc q|_{\lc u\rc_{\rm p}\lc v\rc_{\rm p}}\rc_\rp \lc w\rc_\rp $}&$\lc q|_{\lc u\rc^2_\rp}\rc_\rp \lc v\rc_{\rm p}$\\
\hline
4&$\lc q|_{\lc u\rc_{\rm p}\lc v\rc_{\rm p}}\rc^2_\rp$&$\lc q|_{\lc u\rc^2_\rp}\rc^2_\rp$\\
\hline
\end{tabular}
$$
The first case $3\wedge 3$ has been verified in~\cite[Theorem~4.1]{Qiu14}.  We now check $3\wedge 4$. The others are similar.
 Let
\begin{eqnarray*}
&&f:=\phi_3(z,v)=\lc z\rc_\rp\lc v\rc_{\rm p}-\lc z\lc v\rc_{\rm p}\rc_\rp-\lc\lc z \rc_\rp v\rc_\rp-\lambda \lc zv\rc_\rp\\
&&g:=\phi_4(u,w)=\lc u\rc^2_\rp+\lambda \lc u\rc_{\rm p}.
\end{eqnarray*}
By the monomial order $\leq_{\db}$, we obtain $\bar{f}=\lc z\rc_\rp\lc v\rc_{\rm p}$ and $\bar{g}=\lc u\rc^2_\rp$.
\noindent

\smallskip
{\bf(The case of intersection compositions)}. Suppose that $\omega=\bar{f}\mu=\nu \bar{g}$  with $\max\{\,|\bar{f}|,\,|\bar{g}|\,\}<|\omega| < |\bar{f}|+ |\bar{g}|$. Then $2<|\omega|<3$, contradiction. Thus there are no  intersection compositions of $f$ and $g$.

{\bf(The case of including compositions)}. Suppose that $\omega=\bar{f}= q'|_{\bar{g}}$. This gives $\lc z\rc_\rp\lc v\rc_{\rm p}=q'|_{\lc u\rc^2_\rp}$.  Then $q'=\lc q\rc_\rp \lc v\rc_{\rm p}$ or $q'=\lc z\rc_\rp \lc q\rc_\rp$, where $q$ is in $\frakc^\star(Y)$. Thus $z=q|_{\lc u\rc^2_\rp}$  or $v=q|_{\lc u\rc^2_\rp}$. Since the multiplication of $\bfk\frakc(Y)$ is commutative, we only check the fist case when $\omega=\lc q|_{\lc u\rc^2_\rp}\rc_\rp\lc v\rc_{\rm p}$ and $q'=\lc q\rc_\rp\lc v\rc_{\rm p}$.
Then
\begin{eqnarray*}
(f,g)_\omega^{q'}&=&f-q'|_g\\
&=& \lc q|_{\lc u\rc^2_\rp}\rc_\rp \lc v\rc_{\rm p}-\lc q|_{\lc u\rc^2_\rp}\lc v\rc_{\rm p}\rc_\rp-\lc \lc q|_{\lc u\rc^2_\rp}\rc_\rp v\rc_\rp\\
&&-\lambda \lc q|_{\lc u\rc^2_\rp}v\rc_\rp-\lc q|_{\lc u\rc^2_\rp+\lambda \lc u\rc_{\rm p}}\rc_\rp \lc v\rc_{\rm p}\\
  &=&-\lc q|_{\lc u\rc^2_\rp}\lc v\rc_{\rm p}\rc_\rp-\lc \lc q|_{\lc u\rc^2_\rp}\rc_\rp v\rc_\rp-\lambda \lc q|_{\lc u\rc^2_\rp }v\rc_\rp-\lambda \lc q|_{\lc u\rc_{\rm p}}\rc_\rp\lc v\rc_{\rm p}\\
  &\equiv&\lambda \lc q|_{\lc u\rc_{\rm p}}\lc v\rc_{\rm p}\rc_\rp+\lambda \lc \lc q|_{\lc u\rc_{\rm p}}\rc_\rp v\rc_\rp+\lambda^{2}\lc q|_{\lc u\rc_{\rm p}}v\rc_\rp\\
  &&-\lambda \lc q|_{\lc u\rc_{\rm p}}\lc v\rc_{\rm p}\rc_\rp-\lambda \lc \lc q|_{\lc u\rc_{\rm p}}\rc_\rp v\rc_\rp-\lambda^{2}\lc q|_{\lc u\rc_{\rm p}}v\rc_\rp\\
  &\equiv& 0\; \mod(S_{\rm rb},\omega).
\end{eqnarray*}
Thus $f,g)_\omega^{q'}$ is trivial modulo $(S_{\rm rb},\omega)$.
\noindent

\smallskip
(\mref{it:frb})
By Proposition~\mref{prop:drb}~(\mref{it:fqrb}), $\bfk\frakc(Y)/\Id(S_{\rm rb})$ is the  free  commutative quasi-idempotent Rota-Baxter algebra on $Y$. Then by Item~(\mref{it:gsrb}) and Theorem~\mref{thm:cdl}~(\mref{it:gsb}),  $\Irr(S_{\rm rb})$ is a linear basis of $\bfk\frakc(Y)/\Id(S_{\rm rb})$.
\end{proof}

\subsection{Linear bases of free commutative quasi-idempotent differential Rota-Baxter algebras}
This section will develop Gr\"obner-Shirshov bases for free commutative quasi-idempotent differential Rota-Baxter algebra. For this, we take $\Omega:=\{\lc\,\rc_{\rm d}, \lc\,\rc_{\rm P}\}$. Now we are ready for the main theorem of this paper.
\begin{theorem}Let $Y$ be a well-ordered set. Let
$$ S_{\rm drb}:=S_{\rm d}\cup S_{\rm rb}\cup \Big\{\,\phi_5(u,v):=\lc \lc u\rc_{\rm p}\rc_\rd-u\,|\,u,v\in \frakc(Y)\Big\}.
$$
where $S_{\rm d}$ and $S_{\rm rb}$ be as in Theorem~\mref{thm:fcqd} and Theorem~\mref{thm:fcqrb}, respectively. With the monomial order $\leq_{\db}$ defined in Eq.~(\mref{eq:mord}),
\begin{enumerate}
\item $S_{\rm drb}$ is a Gr\"obner-Shirshov basis in $\bfk\frakc(Y)$.\mlabel{it:gsdrb}
\item $$\Irr(S_{\rm drb}):=\frakc(Y)\Big\backslash\,
\left\{\begin{array}{lll}q_1|_{\lc u\rc_{\rm p}\lc v\rc_{\rm p}}, \,q_2|_{\lc u\rc^2_\rp},\,q_3|_{\lc u\rc_{\rm d}\lc v\rc_{\rm d}},\\
q_4|_{\lc u\rc^2_\rd},\,q_5|_{\lc \lc u\rc_{\rm p}\rc_\rd}
 \end{array}\bigg|q_i\in\frakc^\star(Y),1\leq i\leq 5,u,v\in \frakc(Y)\right\}.$$
is a linear basis of the free  commutative quasi-idempotent differential Rota-Baxter algebra on $Y$.\mlabel{it:fdrb}
\end{enumerate}
\mlabel{thm:drb}
\end{theorem}
\begin{proof}
(\mref{it:gsdrb}) Denote by $i\wedge j$ the composition of $\Omega$-polynomials of types $\phi_i(u,v), \phi_j(u,v)\in S_{\rm drb}$ for $i,j=1,2,3,4,5$.
All ambiguities of  compositions of commutative $\Omega$-polynomials in $S$ are given in the following table.
  \begin{table}[!htbp]
\centering
\begin{tabular}{|c|c|c|c|c|c|}
\hline
\diagbox{i}{$i\wedge j$}{j}&1&2&3&4&5\\
\hline
1&\makecell{$\lc u\rc_{\rm d}\lc v\rc_{\rm d}\lc w\rc_{\rm d}$\\
$\lc q|_{\lc u\rc_{\rm d}\lc v\rc_{\rm d}}\rc_\rd\lc w\rc_{\rm d}$}&$\lc q|_{\lc u\rc^2_\rd}\rc_\rd\lc v\rc_{\rm d}$&$\lc q|_{\lc u\rc_{\rm p}\lc v\rc_{\rm p}}\rc_\rd\lc w\rc_{\rm d}$&$\lc q|_{\lc u\rc^2_\rp}\rc_\rd\lc v\rc_{\rm d}$&$\lc q|_{\lc \lc u\rc_{\rm p}\rc_\rd}\rc_\rd\lc v\rc_{\rm d}$\\
\hline
2&$\lc q|_{\lc u\rc_{\rm d}\lc v\rc_{\rm d}}\rc^2_\rd$&$\lc q|_{\lc u\rc^2_\rd}\rc^2_\rd$&$\lc q|_{\lc u\rc_{\rm p}\lc v\rc_{\rm p}}\rc^2_\rd$&$\lc q|_{\lc u\rc^2_\rp}\rc^2_\rd$&$\lc q|_{\lc \lc u\rc_{\rm p}\rc_\rd}\rc^2_\rd$\\
\hline
3&$\lc q|_{\lc u\rc_{\rm d}\lc v\rc_{\rm d}}\rc_\rp\lc w\rc_\rp$&$\lc q|_{\lc u\rc^2_\rd}\rc_\rp\lc v\rc_{\rm p}$& \makecell{$\lc u\rc_{\rm p}\lc v\rc_{\rm p}\lc w\rc_\rp$\\
$\lc q|_{\lc u\rc_{\rm p}\lc v\rc_{\rm p}}\rc_\rp\lc w\rc_\rp$}&$\lc q|_{\lc u\rc^2_\rp}\rc_\rp\lc v\rc_{\rm p}$&$\lc q|_{\lc\lc u\rc_{\rm p}\rc_\rd}\rc_\rp\lc v\rc_{\rm p}$\\
\hline
4&$\lc q|_{\lc u\rc_{\rm d}\lc v\rc_{\rm d}}\rc^2_\rp$&$\lc q|_{\lc u\rc^2_\rd}\rc^2_\rp$&$\lc q|_{\lc u\rc_{\rm p}\lc v\rc_{\rm p}}\rc^2_\rp$&$\lc q|_{\lc u\rc^2_\rd}\rc^2_\rp$&$\lc q|_{\lc \lc u\rc_{\rm p}\rc_\rd}\rc^2_\rp$\\
\hline
5&$\lc \lc q|_{\lc u\rc_{\rm d}\lc v\rc_{\rm d}}\rc_\rp\rc_\rd$&$\lc \lc q|_{\lc u\rc^2_\rd}\rc_\rp\rc_\rd$&$\lc \lc q|_{\lc u\rc_{\rm p}\lc v\rc_{\rm p}}\rc_\rp\rc_\rd$&$\lc \lc q|_{\lc u\rc^2_\rp}\rc_\rp\rc_\rd$&$\lc \lc q|_{\lc \lc u\rc_{\rm p}\rc_\rd}\rc_\rp\rc_\rd$\\
\hline
\end{tabular}
\end{table}
%Here $x,y,z\in \frakc(Y)$, $q\in \frakc^{\star}(Y)$.
 Firstly, by the proof of Theorem~\mref{thm:d} and Theorem~\mref{thm:rb}, the eight cases: $1\wedge 1$, $1\wedge 2$, $2\wedge 1$, $2\wedge 2$, $3\wedge 3$, $3\wedge 4$, $4\wedge 3$ and $4\wedge 4$, have been proved. By the proof of ~\cite [Theorem~6.1 ]{Qiu14}, the three cases: $3\wedge 5$, $5\wedge 3$ and $5\wedge 5$,  are already validated. We need to check the remaining fourteen cases. We only verify  two cases:$1\wedge3, 2\wedge3$. The others are similar.   Let
\begin{eqnarray*}
&&f_1:=\phi_1(u_1,v_1)=\lc u_1\rc_\rd\lc v_1\rc_\rd+\lambda^{-1}\lc u_1\rc_\rd v_1+\lambda^{-1}u_1\lc v_1\rc_\rd-\lambda^{-1}\lc u_1v_1\rc_\rd\\
&&f_2:=\phi_2(u_2,v_2)=\lc u_2\rc^2_\rd+\lambda^{-1} \lc u_2\rc_\rd\\
&&f_3:=\phi_3(u_3,v_3)=\lc u_3\rc_\rp \lc v_3\rc_\rp-\lc u_3\lc v_3\rc_\rp\rc_\rp-\lc \lc u_3\rc_\rp v_3\rc_\rp-\lambda \lc u_3v_3\rc_\rp\\
&&f_4:=\phi_4(u_4,v_4)=\lc u_4\rc^2_\rp+\lambda \lc u_4\rc_\rp\\
&&f_5:=\phi_5(u_5,v_5)=\lc \lc u_5\rc_\rp \rc_\rd-u_5.
\end{eqnarray*}
By the monomial order $\leq_{\db}$, we obtain $\bar{f_1}=\lc u_1\rc_\rd\lc v_1\rc_\rd$, $\bar{f_2}=\lc u_2\rc^2_\rd$, $\bar{f_3}=\lc u_3\rc_\rp \lc v_3\rc_\rp$, $\bar{f_4}=\lc u_4\rc^2_\rp $ and $\bar{f_5}=\lc \lc u_5\rc_\rp\rc_\rd$.
\noindent

\smallskip
{\bf Case: $1\wedge 3$}.

{\bf(The case of intersection compositions)}. Suppose that $\omega=\bar{f_1}\mu=\nu \bar{f_3}$  with $\max\{\,|\bar{f_1}|,\,|\bar{f_3}|\,\}<|\omega| < |\bar{f_1}|+ |\bar{f_3}|$. Then $|\omega|=3$, and so $\omega=\lc u_1\rc_\rd\lc v_1\rc_\rd\mu=\nu \lc u_3\rc_\rp \lc v_3\rc_\rp$. Thus $\lc v_1\rc_\rd=\lc u_3\rc_\rp$, contradiction.  Thus there are no  intersection compositions of $f_1$ and $f_3$.
\noindent

\smallskip
{\bf(The case of including compositions)}. Rewrite $f_1=\phi_1(z,w)$ and $f_3=\phi_3(u,v)$. Suppose that $\omega=\bar{f_1}= q'|_{\bar{f_3}}$. This gives $\lc z\rc_{\rm d}\lc w\rc_{\rm d}=q'|_{\lc u\rc_{\rm p}\lc v\rc_{\rm p}}$.  Then $q'=\lc q\rc_\rd\lc w\rc_{\rm d}$ or $q'=\lc z\rc_{\rm d}\lc q\rc_\rd$, where $q$ is in $\frakc^\star(Y)$. Thus $z=q|_{\lc u\rc_{\rm p}\lc v\rc_{\rm p}}$  or $w=q|_{\lc u\rc_{\rm p}\lc v\rc_{\rm p}}$. We only check the fist case when $\omega=\lc q|_{\lc u\rc_{\rm p}\lc v\rc_{\rm p}}\rc_\rd\lc w\rc_{\rm d}$ and $q'=\lc q\rc_\rd\lc w\rc_{\rm d}$.
  {\allowdisplaybreaks
    \begin{eqnarray*}
(f_1,f_3)^{q'}_\omega&=&f_1- q'|_{f_3}\\
&=&\lc q|_{\lc u\rc_{\rm p}\lc v\rc_{\rm p}}\rc_\rd \lc w\rc_{\rm d}+\lambda^{-1}\lc q|_{\lc u\rc_{\rm p}\lc v\rc_{\rm p}}\rc_\rd w+\lambda^{-1}q|_{\lc u\rc_{\rm p}\lc v\rc_{\rm p}}\lc w\rc_{\rm d}
  -\lambda^{-1}\lc q|_{\lc u\rc_{\rm p}\lc v\rc_{\rm p}}w\rc_\rd\\
  &&-\lc q|_{\lc u\rc_{\rm p}\lc v\rc_{\rm p}-\lc u\lc v\rc_{\rm p}\rc_\rp- \lc \lc u\rc_{\rm p}v\rc_\rp-\lambda \lc uv\rc_\rp}\rc_\rd\lc w\rc_{\rm d}\\
  &=&\lambda^{-1}\lc q|_{\lc u\rc_{\rm p}\lc v\rc_{\rm p}}\rc_\rd w+\lambda^{-1}q|_{\lc u\rc_{\rm p}\lc v\rc_{\rm p}}\lc w\rc_{\rm d}-\lambda^{-1}\lc q|_{\lc u\rc_{\rm p}\lc v\rc_{\rm p}}w\rc_\rd\\
  &&+\lc q|_{\lc u\lc v\rc_{\rm p}\rc_\rp}\rc_\rd\lc w\rc_{\rm d}+\lc q|_{\lc \lc u\rc_{\rm p}v\rc_\rp}\rc_\rd\lc w\rc_{\rm d}+\lambda \lc q|_{\lc uv\rc_\rp}\rc_\rd\lc w\rc_{\rm d}\\
  &\equiv&\lambda^{-1}\lc q|_{\lc u\lc v\rc_{\rm p}\rc_\rp}\rc_\rd w+\lambda^{-1}\lc q|_{\lc \lc u\rc_{\rm p}v\rc_\rp}\rc_\rd w+\lc q|_{\lc uv\rc_\rp}\rc_\rd w\\
  &&+\lambda^{-1}q|_{\lc u\lc v\rc_{\rm p}\rc_\rp}\lc w\rc_{\rm d}+\lambda^{-1}q|_{\lc \lc u\rc_{\rm p}v\rc_\rp}\lc w\rc_{\rm d}+q|_{\lc uv\rc_\rp}\lc w\rc_{\rm d}\\
  &&-\lambda^{-1}\lc q|_{\lc u\lc v\rc_{\rm p}\rc_\rp}w\rc_\rd-\lambda^{-1}\lc q|_{\lc \lc u\rc_{\rm p}v\rc_\rp}w\rc_\rd-\lc q|_{\lc uv\rc_\rp}w\rc_\rd\\
  &&-\lambda^{-1}\lc q|_{\lc u\lc v\rc_{\rm p}\rc_\rp}\rc_\rd w-\lambda^{-1}q|_{\lc u\lc v\rc_{\rm p}\rc_\rp}\lc w\rc_{\rm d}+\lambda^{-1}\lc q|_{\lc u\lc v\rc_{\rm p}\rc_\rp}w\rc_\rd\\
  &&-\lambda^{-1}\lc q|_{\lc \lc u\rc_{\rm p}v\rc_\rp}\rc_\rd w-\lambda^{-1}q|_{\lc \lc u\rc_{\rm p}v\rc_\rp}\lc w\rc_{\rm d}+\lambda^{-1}\lc q|_{\lc \lc u\rc_{\rm p}v\rc_\rp}w\rc_\rd\\
  &&-\lc q|_{\lc uv\rc_\rp}\rc_\rd w-q|_{\lc uv\rc_\rp}\lc w\rc_{\rm d}+\lc q|_{\lc uv\rc_\rp}w\rc_\rd\\
  &\equiv&0\quad\mod(S_{\rm drb},\omega).
  \end{eqnarray*}}
%%%%%%%%%%%%%%%%%%%%%%%%%%%%%%%%%%%%%%%%%%%%%%%%%%%%%
{\bf Case: $2\wedge 3$}.

{\bf(The case of intersection compositions)}. Suppose that $\omega=\bar{f_2}\mu=\nu \bar{f_3}$  with $\max\{\,|\bar{f_2}|,\,|\bar{f_3}|\,\}<|\omega| < |\bar{f_2}|+ |\bar{f_3}|$. Then $2<|\omega|<3$,  contradiction.  Thus there are no  intersection compositions of $f_2$ and $f_3$.
\noindent

\smallskip
{\bf(The case of including compositions)}. Rewrite $f_2=\phi_2(z,w)$ and $f_3=\phi_3(u,v)$. Suppose that $\omega=\bar{f_2}= q'|_{\bar{f_3}}$. This gives $\lc z\rc^2_\rd=q'|_{\lc u\rc_{\rm p}\lc v\rc_{\rm p}}$.   Thus $q'=\lc q\rc^2_\rd$ and $z=q|_{\lc u\rc_{\rm p}\lc v\rc_{\rm p}}$, where $q$ is in $\frakc^\star(Y)$. Hence $\omega=\lc q|_{\lc u\rc_{\rm p}\lc v\rc_{\rm p}}\rc^2_\rd$.
  \begin{eqnarray*}
(f_2,f_3)^{q'}_\omega&=&f_2-q'|_{f_3}\\
&=&\lc q|_{\lc u\rc_{\rm p}\lc v\rc_{\rm p}}\rc^2_\rd+\lambda^{-1}\lc q|_{\lc u\rc_{\rm p}\lc v\rc_{\rm p}}\rc_\rd
  -\lc q|_{\lc u\rc_{\rm p}\lc v\rc_{\rm p}-\lc u\lc v\rc_{\rm p}\rc_\rp-\lc \lc u\rc_{\rm p}v\rc_\rp-\lambda \lc uv\rc_\rp}\rc^2_\rd\\
  &=&\lambda^{-1}\lc q|_{\lc u\rc_{\rm p}\lc v\rc_{\rm p}}\rc_\rd+\lc q|_{\lc u\lc v\rc_{\rm p}\rc_\rp}\rc^2_\rd+\lc q|_{\lc \lc u\rc_{\rm p}v\rc_\rp}\rc^2_\rd\\
  &\equiv&\lambda^{-1}\lc q|_{\lc u\lc v\rc_{\rm p}\rc_\rp}\rc_\rd+\lambda^{-1}\lc q|_{\lc \lc u\rc_{\rm p}v\rc_\rp}\rc_\rd+\lc q|_{\lc uv\rc_\rp}\rc_\rd\\
  &&-\lambda^{-1}\lc q|_{\lc u\lc v\rc_{\rm p}\rc_\rp}\rc_\rd-\lambda^{-1}\lc q|_{\lc \lc u\rc_{\rm p}v\rc_\rp}\rc_\rd-\lc q|_{\lc uv\rc_\rp}\rc_\rd\\
  &\equiv&0\quad\mod(S_{\rm drb},\omega).
  \end{eqnarray*}
\noindent

\smallskip
(\mref{it:fdrb}) By Proposition~\mref{prop:drb}~(\mref{it:fqdrb}), $\bfk\frakc(Y)/\Id(S_{\rm drb})$ is the free  commutative quasi-idempotent Rota-Baxter algebra on $Y$. Then by Item~(\mref{it:gsrb}) and Theorem~\mref{thm:cdl}~(\mref{it:gsb}),  $\Irr(S_{\rm drb})$ is a linear basis of $\bfk\frakc(Y)/\Id(S_{\rm drb})$.
\end{proof}
\noindent {\bf Acknowledgements}: This work was supported by the National Natural Science Foundation of
China (Grant Nos.~11601199 and 11961031) and  Jiangxi Provincial Natural Science Foundation (Grant No.~20224BAB201003).

%We thank the  referee for very helpful suggestions.

%\noindent
%{\bf Declaration of interests. } The authors declare that they have no known competing financial interests or personal relationships that could have appeared to influence the work reported in this paper.

%\noindent
%{\bf Data availability. } Data sharing is not applicable to this article as no new data were created or analyzed in this study.

\end{document}